\documentclass{amsart}
\usepackage{amscd,amssymb}
\usepackage[alphabetic,backrefs,lite]{amsrefs}
\usepackage[all]{xy}

\newcommand{\defi}[1]{\textsf{#1}}

\newtheorem{theorem}{Theorem}[section]
\newtheorem{lemma}[theorem]{Lemma}
\newtheorem{prop}[theorem]{Proposition}

\newtheorem{cor}[theorem]{Corollary}
\newtheorem{conjecture}[theorem]{Conjecture}

\theoremstyle{definition}
\newtheorem{definition}[theorem]{Definition}
\theoremstyle{remark}
\newtheorem{remark}[theorem]{Remark}

         \newcommand{\cO}{\mathcal O}
\newcommand{\cC}{\mathcal C}

         \newcommand{\cS}{\mathcal S}
\newcommand{\cG}{\mathcal G}

\newcommand{\cM}{\mathcal M}

\newcommand{\bP}{\mathbb P}

\newcommand{\bZ}{\mathbb Z}

\newcommand{\Pic}{{\rm Pic}}
\newcommand{\Cox}{{\rm Cox}}
\newcommand{\Tor}{{\rm Tor}}
\newcommand{\free}{sweeping}

\newcommand{\HX}[2]{{\rm H}^{#1} \big( X, \cO_X(#2) \big)}

\newcommand{\hX}[2]{h^{#1} \big( X, \cO_X(#2) \big)}

\begin{document}

\title
{Cox Rings of degree one del Pezzo surfaces}
\thanks{The first author was partially supported by Jacobs University Bremen, DFG grant STO-299/4-1 and EPSRC grant number EP/F060661/1; the third author is partially supported by NSF grant DMS-0802851.}
\author{Damiano Testa}
\author{Anthony V\'arilly-Alvarado}
\author{Mauricio Velasco}

\address{Mathematical Institute, 24-29 St Giles, Oxford OX1 3LB, United Kingdom}
\email{adomani@gmail.com}

\address{Department of Mathematics, Rice University, MS 136, Houston, Texas 77005, USA}
\email{varilly@rice.edu}
\urladdr{http://math.rice.edu/~av15}

\address{Department of Mathematics, University of California, 
	Berkeley, CA 94720, USA}
\email{velasco@math.berkeley.edu}
\urladdr{http://math.berkeley.edu/\~{}velasco}
\date{\today}

\begin{abstract}
Let $X$ be a del Pezzo surface of degree one over an algebraically closed field, and let $\Cox(X)$ be its total coordinate ring. We prove the missing case of a conjecture of Batyrev and Popov, which states that $\Cox(X)$ is a quadratic algebra.  We use a complex of vector spaces whose homology determines part of the structure of the minimal free $\Pic(X)$-graded resolution of $\Cox(X)$ over a polynomial ring.  We show that sufficiently many Betti numbers of this minimal free resolution vanish to establish the conjecture. 
\end{abstract}

\maketitle

\section{Introduction}

Let $k$ be an algebraically closed field and let $X$ be a smooth projective integral scheme over $k$. Assume that the Picard group ${\rm Pic}(X)$ is freely generated by the classes of divisors $D_0, D_1,\ldots, D_r$.  The \defi{total coordinate ring}, or \defi{Cox ring} of $X$ with respect to this basis is given by
\[
\Cox(X) := \bigoplus_{(m_0,\ldots,m_r)\in\bZ^{r+1}}\HX{0}{m_0 D_0+\dots+m_r D_r},
\]
with multiplication induced by multiplication of functions in $k(X)$. Different choices of bases yield (noncanonically) isomorphic Cox rings.

The first appearance of Cox rings was in the context of toric varieties in~\cite{Cox}.  In that paper Cox proves that if $X$ is a toric variety then its total coordinate ring is a finitely generated 
multigraded polynomial ring, and that $X$ can be recovered as a quotient of an open subset of ${\rm Spec(Cox}(X))$ by the action of a torus. 

Cox rings are finitely generated $k$-algebras in several other cases, including del Pezzo surfaces ~\cite{BP}, 
rational surfaces with big anticanonical divisor~\cite{big}, blow-ups of $\mathbb{P}^n$ at points lying on a rational normal curve~\cite{CT} and wonderful varieties~\cite{BRION}.  All these varieties are examples of {\it Mori dream spaces}~\cite{HK}, and for this class the Cox ring of $X$ captures much of the birational geometry of the variety. For example, the effective and nef cones of $X$ are finitely generated polyhedral cones and there are only finitely many varieties isomorphic to $X$ in codimension one, satisfying certain mild restrictions.  
To add to the long list of consequences of the landmark paper~\cite{BCHM}, we note that log Fano varieties are also Mori dream spaces.

Colliot-Th\'el\`ene and Sansuc introduced universal torsors to aid the study of the Hasse principle and weak approximation on an algebraic variety $X$ over a number field \cites{CTSI, CTSII}; see also~\cite{CTSSD}.  If the Cox ring of $X$ is finitely generated, then the universal torsor of $X$ is an open subset of ${\rm Spec}(\Cox(X))$, an affine variety for which explicit presentations have been calculated in many cases~\cites{HT, Hassett}. Starting with Salberger in~\cite{Sal98}, universal torsors have been successfully applied to the problem of counting points of bounded height on many classes of varieties. The expository article~\cite{Peyre} has a very readable account of the ideas involved.  The explicit descriptions of universal torsors via Cox rings have led to explicit examples of generalized del Pezzo surfaces that satisfy Manin's conjectures for points of bounded height~\cites{dlBB,dlBBD}.

In~\cite{BP}, Batyrev and Popov systematically study the Cox rings of del Pezzo surfaces $X/k$. They show that $\Cox(X)$ is a finitely generated $k$-algebra; moreover, if $\deg(X) \leq 6$ then $\Cox(X)$ is generated by sections whose classes have anticanonical degree one~\cite{BP}*{Theorem~3.2}. Let $\cG$ denote a minimal set of homogeneous generators of $\Cox (X)$, and denote by $k[\cG]$ the polynomial ring whose variables are indexed by the elements of $\cG$.  As a result, $\Cox (X)$ is a quotient of $k[\cG]$, 
\[ \Cox(X) \cong k[\cG]/I_X. \]

Batyrev and Popov provided a conjectural description of the generators of $I_X$.

\begin{conjecture}[Batyrev and Popov]
\label{C:BP}
Let X be a del Pezzo surface. The ideal $I_X$ is generated by quadrics.
\end{conjecture}

Quadratic relations  have a clear geometric meaning: linear systems associated to degree two nef divisors on $X$ have many more reducible elements than their dimension; thus there are linear 
dependence relations among degree two monomials in $k[\cG]$.  All quadratic relations 
arise in this way (see \S\ref{ulti}).

\begin{remark}
Del Pezzo surfaces of degree at least six are toric varieties and hence their Cox rings are polynomial rings. The Cox ring of a del Pezzo surface of degree five is the homogeneous coordinate ring of the Grassmannian $Gr(2,5)$~\cite{BP}.
\end{remark}

Several partial results are known about this conjecture: Stillman, together with the first and third authors established it for del Pezzo surfaces of degree four, and general del Pezzo surfaces of degree three in~\cite{STV}. The third author and Antonio Laface gave a proof of the conjecture in~\cite{LV} for surfaces of degree at least two.  Recently, Serganova and Skorobogatov established Conjecture~\ref{C:BP} up to radical for surfaces of degree at least two, using representation theoretic methods, in~\cites{SS,SS2} (see also \cite{Derenthal} for related work).   In all cases, the conjecture for degree one surfaces eluded proof. The purpose of this paper is to fill in this gap for surfaces defined over a field of characteristic different from two, as well as for certain general surfaces in characteristic two that we call sweeping (see Definition~\ref{sweeping}).

\begin{theorem}
\label{T:main theorem}
Let $X$ be a del Pezzo surface of degree one; if the characteristic of $k$ is two, then assume that $X$ is {\free}. The ideal $I_X$ is generated by quadrics.
\end{theorem}

\begin{remark}
Our method of proof is cohomological in nature and relies mainly on the Kawamata-Viehweg vanishing theorem.  By~\cite{Te} and~\cite{Xie}, this theorem is independent of the characteristic of $k$ for rational surfaces.
\end{remark}

The argument we give works for del Pezzo surfaces of every degree.  We include a short proof of the conjecture for del Pezzo surfaces of degree at least two in \S\ref{Xgrado2}.  An alternative proof of Conjecture~\ref{C:BP} for general del Pezzo surfaces in characteristic zero, obtained independently by computational means, appears in a paper by Sturmfels and Xu~\cite{SX}.  They apply the theory of sagbi bases to construct an initial toric ideal of $I_X$ which is generated by quadrics. Geometrically this corresponds to a
degeneration of the universal torsor on X to a suitable toric variety.

In order to prove Theorem~\ref{T:main theorem}, we modify the approach taken in~\cite{LV}. We use a complex of vector spaces whose homology determines part of the structure of the minimal free $\Pic(X)$-graded resolution of $\Cox(X)$ over a polynomial ring.  We show that sufficiently many Betti numbers of this minimal free resolution vanish to establish the theorem. We hope that similar techniques can be applied to obtain presentations of Cox rings of other classes of varieties, for example singular del Pezzo surfaces, blow-ups of $\bP^n$ at points lying on the rational normal curve of degree $n$ and $\overline M _{0,n}$.

The paper is organized as follows.  In \S\ref{S:notation} we define del Pezzo surfaces, fix presentations for their Cox rings and establish notation for the rest of the paper. In \S\ref{S:geometry} we study the nef cone of del Pezzo surfaces and prove some geometric results.  In \S\ref{pain} we analyze low degree linear systems on del Pezzo surfaces of degree at most two.  Our proof of Theorem~\ref{T:main theorem} uses the statement of~\cite{BP}*{Proposition~3.4}.  However, the proof of this proposition given in~\cite{BP} has a gap: it applies only to general del Pezzo surfaces of degree one in characteristic not two (see~\cite{Popov}).  Thus we prove Proposition~\ref{tritangenti} in order to establish Theorem~\ref{T:main theorem} for all del Pezzo surfaces of degree one in characteristic not two.  In \S\ref{S:betti} we review the approach of~\cite{LV} to study the ideal of relations of finitely generated Cox rings, adapting it to the case of degree one del Pezzo surfaces. In \S\ref{S:cap and stop} we define the notions of \defi{capturability} of a divisor and \defi{stopping criterion}, and prove Theorem~\ref{thm:games}, the main ingredients in our proof of the Batyrev--Popov conjecture. In \S\ref{S:ample four} we show the capturability of most divisors on del Pezzo surfaces of degree one. We then handle the remaining cases for del Pezzo surfaces of degree one (\S\S\ref{S:ample three},\,\ref{S:non ample divisors}) and del Pezzo surfaces of higher degree (\S\ref{Xgrado2}). In \S\ref{genedue} we finish the proof of the Batyrev--Popov conjecture and give the first multigraded Betti numbers of the Cox rings of del Pezzo surfaces.

\subsection*{Acknowledgements}

We thank Bernd Sturmfels and David Eisenbud for many useful discussions and Ragni Piene for valuable conversations on the Gauss map
for curves in positive characteristic while at the Spring 2009 MSRI's Algebraic Geometry semester. We would also like to thank the anonymous referee for carefully reading the manuscript and for suggesting valuable improvements. 

\section{Notation and background on del Pezzo surfaces}
\label{S:notation}

We briefly review some facts about del Pezzo surfaces and establish much of the paper's notation along the way.

\begin{definition}
A \defi{del Pezzo surface} 
$X$ is a smooth, projective surface over $k$ whose anticanonical divisor $-K_X$ is ample.  
The \defi{degree} of $X$ is the integer $(K_X)^2$.
\end{definition}

\subsection{Picard groups and Cox rings}
A del Pezzo surface $X$ not isomorphic to $\bP^1 \times \bP^1$ is isomorphic to a blow-up of $\bP^2$ centered at $r \leq 8$ points \emph{in general position}: this means no three points on a line, no six on a conic and no eight on a singular cubic with a singularity at one of the points. Let $L$ be the inverse image of a line in $\bP^2$ and let $E_1,\dots,E_r$ be the exceptional divisors corresponding to blown-up points. Then $\bigl(L,E_1,\dots,E_r)$ is a basis for $\Pic (X)$, and 
\[
\Cox(X) := \hspace{-5pt} \bigoplus_{(m_0,\dots,m_r)\in\mathbb{Z}^{r+1}} \hspace{-5pt} 
\HX{0}{m_0L+m_1E_1+\dots+m_rE_r}.
\]
If $X$ has degree one, then with the above choice of basis, the classes in $\Pic(X)$ of the $240$ exceptional curves  are given in Table~\ref{tab:-1curves1} (see~\cite{M}). 

For a del Pezzo surface $X$ of degree at least two, we let $\cG = \cC$ denote the set of exceptional curves.  If $X$ has degree one, then we let $\cC$ denote the set of exceptional curves on $X$ and $\cG : = \cC \cup \{K_1,K_2\}$, where $K_1,K_2 \in |{-K}_X|$ are distinct. In all cases, we write $k[\cG]$ for the polynomial ring whose 
variables are indexed by the elements of $\cG$.  By~\cite{BP}*{Theorem~3.2}, we know that $\Cox(X)$ is a finitely generated 
$k$-algebra, generated by sections whose divisor classes have anticanonical degree one.  
For all $G \in \cG$, we let $g \in \HX{0}{G}$ be a non-zero element.  We typically use uppercase letters for 
divisors on $X$ and we denote the generator of $k[\cG]$ associated to an exceptional 
curve $E$ by the corresponding lowercase letter $e$. If $X$ has degree one, then we denote the generators of $k[\cG]$ 
corresponding to $K_1$ and $K_2$ by $k_1$ and $k_2$, respectively.  Furthermore, given any exceptional 
curve $E \in \cC$, we let $E'$ denote the unique exceptional curve whose divisor class is 
$-2K_X-E$, and we denote accordingly the generator corresponding to $E'$ by $e'$.

\begin{table}[h]
\caption{Exceptional curves on $X$}
\label{tab:-1curves1}
\begin{center}
\begin{tabular}{|c|l|}
\hline
\# of curves & Picard Degree (up to a permutation of $E_1, \ldots , E_8$)\\[2pt]
\hline
\hline
$8$ & $E_1$\\[2pt]
$28$ & $L-E_1-E_2$ \\[2pt]
$56$ & $2L-E_1-E_2-E_3-E_4-E_5$\\[2pt]
$56$ & $3L-2E_1-E_2-E_3-E_4-E_5-E_6-E_7$\\[2pt]
$56$ & $4L - 2E_1 - 2E_2 - 2E_3 - E_4 - E_5 - E_6 - E_7 - E_8$\\[2pt]
$28$ & $5L - 2E_1 - 2E_2 - 2E_3 - 2E_4 - 2E_5 - 2E_6 - E_7 - E_8$\\[2pt]
$8$ & $6L - 3E_1 - 2E_2 - 2E_3 - 2E_4 - 2E_5 - 2E_6 - 2E_7 - 2E_8$\\
\hline 
\end{tabular}
\end{center}
\end{table}
\noindent
There is a surjective morphism 
\[
k[\cG] \to \Cox(X)
\]
that maps $g$ to $g$. We denote the kernel of this map by $I_X$.

For any divisor $D$ on $X$, any integer $n$ and any ${\rm Pic} (X)$-graded ideal 
$J \subset k[\cG]$, denote by $J_D$ the vector space of homogeneous elements of ${\rm Pic} (X)$-degree 
$D$, by $J_n$ the vector space of the homogeneous elements of anticanonical degree 
$n$ and call them respectively the degree $D$ part of $J$ and the degree $n$ part of $J$.

Finally, let $J_X$ be the ideal generated by $(I_X)_2$; since $J_X\subset I_X$, there is a surjection
\[
(k[\cG]/J_X)_D\to (k[\cG]/I_X)_D.
\]
This map plays a role in \S\ref{S:ample three}.

\section{Remarks on the nef cone of del Pezzo surfaces}
\label{S:geometry}

In this section we collect basic results on del Pezzo surfaces that we use in the paper.  The following result is well-known (see~\cite{De}*{p.~148~6.5}).

\begin{prop} \label{clane}
Let $X$ be a del Pezzo surface of degree $d \leq 7$.  A divisor 
class $C \in {\rm Pic} (X)$ is nef (respectively ample) if and 
only if $C \cdot E \geq 0$ (respectively $C \cdot E > 0$) for all 
$(-1)-$curves $E \subset X$.
\end{prop}

\begin{definition}
Let $X$ be a del Pezzo surface.  The \defi{minimal ample divisor on $X$} is 
the ample divisor $A_X$ defined in the following table:
\[
\begin{array} {|c|c|c|}
\hline 
X & A_X & (K_X)^2 \\
\hline
\hline
\bP^2 & - \frac{1}{3} K_X & 9 \\[5pt]
\bP^1 \times \bP^1 & - \frac{1}{2} K_X & 8 \\[5pt]
Bl_p(\bP^2) & 2L-E & 8 \\[5pt]
Bl_{p_1 , \ldots , p_r } (\bP^2) & -K_X & 9-r \\[5pt]
\hline 
\end{array}
\]
\end{definition}
We also extend the list to include $X = \bP^1$ and we define the minimal ample 
divisor on $\bP^1$ to be the class of a point. To simplify the notation, if $b \colon X\to Y$ is a morphism, then we may denote $b^{*}(A_Y)$ by $A_Y$.  Minimal ample divisors allow us to give a geometric description of nef divisors.

\begin{cor} \label{maquale}
Let $X_r $ be a del Pezzo surface of degree $9 - r$.  Let $D \in {\rm Pic} (X_r )$ be a nef divisor.  Then we can  find 
\begin{itemize}
\item non-negative integers $n_0, n_1, \ldots , n_r $;
\item a sequence of morphisms 
\[
X_r \longrightarrow X_{r -1} \longrightarrow {\cdots} \longrightarrow
X_1 \stackrel{b}{\longrightarrow} X_{0} ,
\]
where each morphism is the contraction of a $(-1)$-curve except for $b$, which is allowed to be a conic bundle;
\end{itemize}
such that 
\[
D = n_r A_{X_r } + n_{r -1} A_{X_{r -1}} + \ldots +n_0 A_{X_0} .
\]
\end{cor}

\begin{proof}
We proceed by induction on $r$, the cases $r \leq 1$ being immediate. Suppose that $r \geq 2$ and let 
\(n := \min \bigl\{ L \cdot D ~|~ 
L \subset X {\text{ is a $(-1)-$curve} } \} \).  By assumption we have 
$n \geq 0$.  Let $\bar D := D + n K_{X_r}$; note that $\bar D$ 
is nef by Proposition~\ref{clane}.  Choose a $(-1)-$curve  $L_0 \subset X$ such that $\bar D \cdot L_0 = 0$.  
Thus $\bar D$ is the pull-back of a nef divisor on the del Pezzo surface 
$X_{r -1}$ obtained by contracting $L_0$.  The result follows by 
the inductive hypothesis.
\end{proof}

Observe that the integer $n_r$ in the statement of the corollary is the nef threshold (see~\cite{Reid}*{p.~126}),  of the divisor $D$ with respect to the minimal ample divisor of $X_r$.  The minimal ample divisor is minimal in the sense that for every ample divisor $A$ on $X$, the divisor $A-A_X$ is nef: this follows from Corollary~\ref{maquale}.

By~\cite{Ko}*{Proposition~III.3.4} and Corollary~\ref{maquale} we deduce that every nef divisor $N$ on a del Pezzo surface is effective and that $|2N|$ is base-point free. Moreover, if $N$ is a nef non big divisor, then $N$ is a multiple of a conic bundle.

Let $X \to Y$ be a morphism with connected fibers and let $A$ be the pull-back to 
$X$ of the minimal ample divisor on $Y$.  If $Y \simeq \bP^1$, then we call $A$ a 
\defi{conic}.  If $Y \simeq \bP^2$, then we call $A$ a \defi{twisted cubic}, 
by analogy with the case of the cubic surfaces.  Finally, if $Y$ is a del Pezzo surface 
of degree $d$, then we call $A$ an \defi{anticanonical divisor of degree $d$ in $X$}.

We summarize and systematize the previous discussion in the following lemma.

\begin{lemma} 
\label{mad}
Let $X$ be a del Pezzo surface.  The cone of nef divisors on $X$ is generated by 
the following divisors:
\begin{enumerate}
\item the conics; 
\item the twisted cubics; 
\item the anticanonical divisors of degree $d \leq 3$ in $X$.
\end{enumerate}
In particular, if $N$ is any non-zero nef divisor on $X$, then we may find $r:=8-\deg(X)$ distinct 
exceptional curves $E_1 , E_2 , \ldots , E_r$ on $X$ such that $N-E_1, \ldots , N-E_r$ 
are either nef or the sum of a nef divisor and an exceptional curve, unless $\deg (X) = 1$ 
and $N=-K_X$.
\end{lemma}

\begin{proof}
The divisors in the list are clearly nef.  Conversely, if $D$ is any nef divisor, then 
either $D$ is a multiple of a conic, and we 
are done, or a positive multiple of $D$ induces a morphism with connected fibers 
$X \to Y$.  It is clear that if $A$ is the pull-back to $X$ of the minimal ample 
divisor on $Y$, then $D-A$ is nef.  By induction on $n := -K_X \cdot D$ we therefore 
reduce to showing that the anticanonical divisors of degree $d \geq 4$ in $X$ are 
non-negative linear combinations of the divisors listed.  This is immediate.

For the second statement, note that it suffices to check it for the divisors in the list, 
and for $-2K_X$ if $\deg (X) = 1$, where the result is easy to verify.
\end{proof}

The following lemma will be used in the proof of Lemma~\ref{cocco}.

\begin{lemma} 
\label{dose}
Let $X$ be a del Pezzo surface, let $b \colon X \to \bP^1$ be a conic bundle with fiber 
class $Q$ and let $C$ be an exceptional curve such that $C \cdot Q = 2$.  There are 
exactly five reducible fibers of $b$ such that $C$ intersects both components.  In 
particular, if $\deg (X) = 1$, then there are two reducible fibers of $b$ such that $C$ is 
disjoint from one of the two components in each fiber.
\end{lemma}

\begin{remark}
If $\deg (X) \geq 4$, then there are no exceptional curves $C$ such that $C \cdot Q = 2$, 
and hence the lemma applies non-trivially only to the cases $\deg (X) \leq 3$.
\end{remark}

\begin{proof}
Let $S$ and $T$ be the components of a reducible fiber of $b$.  Since $C \cdot Q = 2$, 
there are only two possibilities for the intersection numbers $C \cdot S$ and $C \cdot T$: 
they are either both equal to one, or one of them is zero and the other is two.  Suppose 
that there are $k$ reducible fibers of $b$ such that the curve $E$ intersects both 
components.  Thus contracting all the components in fibers of $b$ disjoint from $E$ and 
one component in each of the remaining reducible fibers, we obtain a relatively minimal 
ruled surface $b' \colon X' \to \bP^1$, together with a smooth rational curve $C'$ in $X'$, 
the image of $C$, having anticanonical degree $k+1$, square $k-1$ and intersection 
number two with a fiber of $b'$.  Since $X'$ is isomorphic to either $\bP^1 \times \bP^1$ 
or $Bl_p(\bP^2)$, a direct calculation shows that this is only possible if 
$X' \simeq \bP^1 \times \bP^1$ and $k = 5$.

The last statement follows from the fact that any conic bundle structure on $X$ contains 
exactly $8-\deg (X)$ reducible fibers.
\end{proof}

\section{Low degree linear systems} \label{pain}

The lemmas in this section determine subsets of monomials that span $\Cox(X)_D$ when $(\deg(X),D) \in \big\{ (2, -K_X), (1,-2K_X) \big\}$.

\begin{lemma} \label{gene2}
Let $X$ be a del Pezzo surface of degree two.  The linear system $|{-K}_X|$ is spanned 
by any five of its reducible elements.
\end{lemma}

\begin{proof}
The linear system $|{-K_X}|$ defines a separable morphism of degree two $\varphi \colon X \to \bP^2$ such that every pair of exceptional curves $E,E'$ on $X$ with $E+E' = -K_X$ maps under $\varphi$ to the same line in $\bP^2$.  To prove the result it suffices to show that any five of these lines have no base-point.  Suppose that $p \in \bP^2$ is 
contained in $k_p \geq 4$ such lines; let $q \in X$ be a point in the inverse image of $p$.  By construction, the point $q$ is contained 
in $k_p$ exceptional curves.  Let $E$ be an exceptional curve containing $q$ and let 
$E'$ be the exceptional curve such that $E+E' = -K_X$.  The curve $E'$ is disjoint from 
the exceptional curves $F$ such that $F \cdot E = 1$, since $F \cdot E' = 0$.  Thus contracting $E'$ we 
obtain a del Pezzo surface $Y$ of degree three with a point contained in at least 
$k_p-1$ exceptional curves.  The anticanonical linear system embeds $Y$ as a smooth cubic surface in $\mathbb{P}^3$, and exceptional curves through a point $y \in Y$ are lines 
contained in the tangent plane to $Y$ at $y$.  This implies that $k_p-1 \leq 3$, and the result follows.
\end{proof}

To study the case of degree one del Pezzo surfaces we begin with a lemma.

\begin{lemma} \label{uti}
Let $R \subset \mathbb{P}^3$ be a curve, let $r\in R$ be a smooth point and let $p \in \mathbb{P}^3$ be a closed point different from $r$.  Let $H$ be a plane through $p$ intersecting $R$ at $r$ with order of contact $m\geq 2$. Let $\pi_R \colon R \to \mathbb{P}^2$ be the projection away from $p$, let $\bar R = \pi_R(R)$, and let $\gamma \colon R\dashrightarrow (\mathbb{P}^2)^\vee$ be the composition of $\pi_R$ and the Gauss map of $\bar R$, i.e., the map sending a general point $q \in R$ to the tangent line to $\bar R$ at $\pi_R (q)$. Then the map $\gamma $ is defined at $r$ and one of the following occurs: 
\begin{enumerate}
\item{the tangent line to $R$ at $r$ contains $p$, or}
\item{the length of the localization at $r$ of the fiber of $\gamma$ at $\gamma(r)$ equals $m-1$ if ${\rm char}(k) \nmid m$ and it equals $m$ if ${\rm char}(k) \mid m$.}
\end{enumerate}
\end{lemma}
\begin{proof}  
The rational map $\gamma $ is defined at $r$ since $r \in R$ is a smooth point and the range of $\gamma$ is projective. Choose homogeneous coordinates $X_0,X_1,X_2,X_3$ on $\mathbb{P}^3$ so that $p=[1 , 0 , 0 , 1]$ and $r=[0 , 0 , 0 , 1]$. In these coordinates $H$ is defined by the linear form $A X_1+B X_2$, for some $A, B \in k$. In the affine coordinates $(x_0 = \frac{X_0}{X_3} , x_1 = \frac{X_1}{X_3} , x_2 = \frac{X_2}{X_3})$ the curve $R$ is defined near the origin $r$ by a complete intersection $(f_1,f_2)$ and thus its embedded tangent space is the kernel of the matrix $DF(r)= \bigl( \frac{\partial f_i}{\partial x_j}(r) \bigr)$ with $i \in \{1 , 2\}$ and $j \in \{0 , 1 , 2\}$.  Therefore either the first column of $DF(r)$ vanishes (i.e.\ $\frac{\partial f_1}{\partial x_0}(r)=\frac{\partial f_2}{\partial x_0}(r)=0$) and the tangent line to $R$ through $r$ contains $p$ or some $2\times 2$ minor of $DF(r)$ containing the first column is non-zero, since otherwise all $2\times 2$ minors of $DF(r)$ vanish contradicting the assumption that $R$ is nonsingular at $r$.  
In the latter case, the completion of the local ring of $R$ at $r$ is isomorphic to $k [\![t]\!]$ via a parametrization in formal power series of the form $(x_0(t),x_1(t),t)$. Since $H$ has order of contact $m$ at $r$, the power series $Ax_1(t)+Bt$ vanishes to order $m$ so $A\neq 0$ and we have
\[
x_1(t)= -\frac{B}{A} t + c_mt^m+ (\text{higher order terms})
\]
with $c_m\neq 0$.  The equation of the tangent line to $\bar R$ at $\pi_R(x_0(t),x_1(t),t,1)$ is 
\[
-(X_1-x_1(t))+(X_2-t)x_1'(t)=0 ,
\]
so the morphism $\gamma$ is given by
\[
(x_0(t),x_1(t),t,1) \longmapsto [-1, x_1'(t) , x_1(t)-tx_1'(t)] .
\]
The localization at $r$ of the fiber of $\gamma $ at $\gamma(r)$ is thus given by the ideal $I=(Ax_1'(t)+B, x_1(t)-tx_1'(t))$.  Since we have 
\begin{eqnarray*}
Ax_1'(t)+B & = & Ac_mmt^{m-1} +(\text{higher order terms}) \\[5pt]
x_1(t)-tx_1'(y) & = & c_m(1-m)t^m+(\text{higher order terms}),
\end{eqnarray*}
it follows that either ${\rm char}(k)\nmid m $ and $I=(t^{m-1})$, or ${\rm char}(k) \mid m$ and $I=(t^m)$ as we wanted to show.
\end{proof}

Let $X$ be a del Pezzo surface of degree one.  The linear system $|{-2K_X}|$ is base-point free and the image of the associated morphism $\kappa \colon X \to |{-2K_X}|^* \simeq \bP^3$ is a cone $W$ over a smooth conic; the morphism $\kappa $ is a double cover of $W$ branched at the cone vertex $w \in W$ and along a divisor $R \subset W$.  We call {\it tritangent planes of $X$} the elements of $|{-2K_X}|$ supported on exceptional curves and think of them as planes in $|{-2K_X}|^* = \bP^3$.  There are $120$ such planes  and they do not contain the vertex $w$ of the cone, since the planes containing $w$ correspond to elements of $|{-2K_X}|$ supported on the sum of two effective anticanonical divisors (see~\cite{Ko}*{Section~III.3} and~\cite{M}*{Chapter~IV} for details).

\begin{definition} \label{sweeping}
A del Pezzo surface $X$ of degree one over a field $k$ is \defi{{\free}} if any $119$ tritangent planes of $X$ span $|{-2K_X}|$.
\end{definition}

Our goal in the remainder of this section is to show that every del Pezzo surface of degree one in characteristic not two is \free; if the characteristic is two, then we only show that a general del Pezzo surface of degree one is sweeping.  In the following sections we prove that if $X$ is a sweeping del Pezzo surface of degree one, then Conjecture~\ref{C:BP} holds for $X$.

From now and until the end of the proof of Proposition~\ref{tritangenti} we assume that the characteristic of $k$ is not two.  In this case the curve $R$ is smooth and it is the complete intersection of $W$ with a cubic surface.  Hence $R$ is a canonical curve of genus four and degree six admitting a unique morphism of degree three to $\mathbb{P}^1$ (up to changes of coordinates on $\mathbb{P}^1$) obtained by projecting away from $w$.  The tritangent planes to $R$ are planes in $\mathbb{P}^3$ not containing $w$ whose intersection with $R$ is twice an effective divisor.

\begin{prop}
\label{tritangenti}
If $X$ is a del Pezzo surface of degree one over a field $k$ of characteristic different from two, then $X$ is {\free}. More precisely, any $113$ of the tritangent planes of $X$ span $|{-2}K_X|$.
\end{prop}

\begin{proof}
Let $p \in \bP^3$ be a closed point.  We say that a tritangent plane $H$ containing $p$ is \defi{$p$-regular} (resp.\ \defi{$p$-singular}) if $p$ does not belong to (resp.\ $p$ belongs to) a tangent line to $R$ at some point in $H\cap R$; denote by $k_p$ the number of tritangent planes containing $p$ and by $k_p^r$ (resp.\ $k_p^s$) the number of $p$-regular (resp.\ $p$-singular) tritangent planes.  Suppose that the point $p$ belongs to $R$; then there are no $p$-regular tritangent planes, since every tritangent plane containing $p$ contains also the tangent line to $R$ at $p$.  Thus if $q$ is any point on the tangent line to $R$ at $p$, then we have $k_p \leq k_q^s \leq k_q$ and hence it suffices to prove the proposition assuming that $p \in \mathbb{P}^3$ is a closed point not in $R$.  Therefore let $p \in \bP^3 \setminus R$ be a closed point; projecting away from $p$ we obtain a morphism $\pi_R \colon R \to \bP^2$; let $\bar R = \pi _R(R)$.  Let $\bar\gamma\colon \bar R \dasharrow (\mathbb{P}^2)^\vee$ be the Gauss map, let $\gamma \colon R \to (\mathbb{P}^2)^\vee$ be the unique morphism that extends the composition $\bar\gamma\circ(\pi_R)$, and let $\check{R} = \gamma(R)$. The morphism $\gamma$ factors as $R \xrightarrow{g} N \xrightarrow{\nu} \check R$, where $\nu\colon N \to \check{R}$ is the normalization of $\check{R}$.  
We summarize these definitions in the following commutative diagram.
\begin{equation} \label{diatutto}
\begin{minipage}{150pt}
\xymatrix{
\bar{R} \ar@{-->}[r]^{\bar\gamma} & \check{R} \\
R \ar[u]^{\pi_R} \ar[ru]^\gamma \ar[r]^g & N \ar[u]_\nu
}
\end{minipage}
\end{equation}
The argument consists in identifying the contributions of the tritangent planes through $p$ in terms of combinatorial invariants of Diagram~\eqref{diatutto}.  More precisely, the $p$-singular tritangent planes correspond (at most $7:1$) to the points of $\bar{R}$ where the morphism $\pi _R$ is ramified; the $p$-regular tritangent planes correspond to points of $\check{R}$ where (the separable part of) $\gamma $ is ramified.

We estimate the number $k_p^s$ of $p$-singular tritangent planes in the following lemma.

\begin{lemma} \label{badtritangenti}
As above, let $R$ be the ramification divisor of $\kappa$.
\begin{enumerate}
\item \label{setri}
If $\ell \subset \mathbb{P}^3$ is a tangent line to $R$, then there are at most 7 tritangent planes containing $\ell$;
\item \label{ksi}
Let $p \in \mathbb{P}^3$ be a closed point not in $R$.  There are at most $42$ $p$-singular tritangent planes.
\end{enumerate}
\end{lemma}

\begin{proof}
\eqref{setri}
Let $r \in R$ be a point such that $\ell $ is tangent to $R$ at $r$ and let $\alpha\colon R\to \bP^1$ be the morphism obtained by considering the pencil of planes in $\bP^3$ containing $\ell $. The morphism $\alpha$ has degree at most four since every plane containing $\ell $ is tangent to $R$ at $r$ and hence intersects $R$ at $r$ with multiplicity at least two. If $\alpha$ were not separable, then the characteristic of $k$ would be three (recall that we assume ${\rm char}(k) \neq 2$), $\alpha $ would be the Frobenius morphism, and the curve $R$ would be rational, which is not the case. We deduce that $\alpha$ is separable; if $\alpha $ has degree three, then $\ell $ contains the vertex $w$ of the cone $W$ and in this case no tritangent plane contains $\ell$. Moreover, $\alpha $ cannot have degree two since $R$ is not hyperelliptic.  Thus $\alpha$ is separable and we reduce to the case in which the degree of $\alpha $ is four. If $H$ is a tritangent plane to $R$ containing $\ell $, then we have $R\cap H=2((r)+(p)+(q))$, and hence the contribution of $H$ to the ramification divisor of $\alpha$ is at least two. From the Hurwitz formula we deduce that the ramification divisor has degree $14$ and $r$ is contained in at most $\frac{14}{2} = 7$ tritangent planes.

\noindent
\eqref{ksi}  First, we reduce to the case that the morphism $\pi _R$ has degree at most two (and in particular it is separable).  Indeed, the degree of $\pi _R$ divides six, and since $R$ is not contained in a plane, the image $\bar R$ cannot be a line; therefore it suffices to analyze the cases in which the degree of $\pi _R$ equals three.  If $\pi_R$ were not separable of degree three, then it would be purely inseparable and $\bar R$ would have degree two but geometric genus four, which is not possible.  If $\pi _R$ is separable of degree three, then $p$ is the vertex $w$ of the cone $W$ and hence $k_p^s = 0$, since there are no tritangent planes through $w$. This completes the reduction. Note that if the degree of $\pi _R$ equals two, then the image of $\pi _R$ is a smooth plane cubic, since $R$ is not hyperelliptic.

Second, the morphism $\pi _R$ ramifies at every tangent line to $R$ through $p$ by~\cite{Hartshorne}*{Proposition~IV.3.4}.

Finally, the image of a ramification point of $\pi _R$ is either a singular point of $\bar{R}$ or a ramification point of the morphism induced by $\pi _R$ to the normalization of $\bar{R}$.  Using the Hurwitz formula, the formula for the arithmetic genus of a plane curve, and the equality $\deg (\bar{R}) \cdot \deg (\pi _R) = 6$, we deduce that the number of points in $\mathbb{P}^2$ corresponding to ramification points of $\pi _R$ is at most six.  Thus there are at most six lines in $\mathbb{P}^3$ containing $p$ such that every $p$-singular tritangent plane contains one of these lines.  We conclude using~\eqref{setri}.
\end{proof}

By Lemma~\ref{badtritangenti}~\eqref{ksi} we have $k_p^s\leq 42$.  Thus to prove the proposition, it suffices to show that $k_p^r \leq 70$.

Define the torsion sheaf $\Delta$ on $\check{R}$ by the exact sequence
\begin{equation}
\label{eq:delta invariant}
0 \to \mathcal{O}_{\check{R}} \to \nu_*\mathcal{O}_N \to \Delta \to 0.
\end{equation}
From the long exact cohomology sequence, we deduce that
\begin{equation}
h^1(N,\mathcal{O}_N) + h^0(\Delta) = h^1(\check{R},\mathcal{O}_{\check{R}}).
\end{equation}
Applying the Hurwitz formula to the composition of $\pi_R$ with a projection of $\bar R$ from a general point of $\mathbb{P}^2$, we obtain $\deg(\check{R}) \leq 18$, and hence $h^1(\check{R},\mathcal{O}_{\check{R}}) \leq 136$.

For each tritangent plane $H$ through $p$ denote by $\ell _H \in \check{R} \subset (\mathbb{P}^2)^\vee$ the point corresponding to the image in $\mathbb{P}^2$ of $H$ under the projection away from $p$.  We want to estimate the contribution of the point $\ell _H \in \check R$ to $h^0(\Delta)$.  Note that this contribution is at least ${\rm length}(\mathcal{O}_{\nu^{-1}(\ell _H)}) - 1$, and that
\begin{equation} \label{lengths}
{\rm length} \bigl( \mathcal{O}_{\nu^{-1}(\ell _H)} \bigr) \cdot \deg(g) ={\rm length} \bigl( \mathcal{O}_{\gamma^{-1}(\ell _H)} \bigr) .
\end{equation}

Factor $g$ as the composition of a power of the Frobenius morphism, followed by a separable morphism $g _s$.  Let $d_i$ be the inseparable degree of $g$ and $d_s = \deg (g_s)$ be the separable degree of $g$, let also $\check{d}$ denote the degree of $\check{R}$; recall that we have $d_i d_s \check{d} \leq 18$.  

Suppose first that $d_i d_s = 1$.  By Lemma~\ref{uti} and~\eqref{lengths} each $p$-regular tritangent plane contributes at least $2$ to $h^0(\Delta)$; hence we have $h^0(\Delta) \geq 2 k_p^r$, and we conclude that $k_p^r \leq 68$.

Suppose now that $d_i d_s > 1$ and hence we have $\check{d} \leq 9$.  By the Hurwitz formula, the degree of the ramification divisor of the morphism $g _s$ is at most $6+2 \deg (g _s) \leq 6+36 = 42$.  The curve $\check{R}$ has degree at most $9$, and hence it has at most $28$ singular points.

Let $U \subset N \times \mathbb{P}^3$ be the universal family of planes through $p$ tangent to $R$ and let $Z := (U \cap N \times R)$.  Thus $Z$ defines a family of closed subschemes of $R \subset \mathbb{P}^3$ of dimension zero and degree six.  The geometric generic fiber of $Z \to N$ therefore determines a partition of $6$ that we call the generic splitting type.  Alternatively, if $H$ is a plane through $p$ tangent to $R$, then $H \cap R$ determines an effective divisor of degree six on $R$; if $H$ is general with the required properties, then the multiplicities of the geometric points of $H \cap R$ are the generic splitting type.

By~\cite{HefezKleiman}*{Theorem~3.5} and \cite{Kaji}*{p.\ 529}, the generic splitting type of $\gamma$ is of the form $(a,\dots,a,1,\dots,1)$, where either $a=2$ and the morphism is separable or $a$ is the inseparable degree of the rational map $\bar R \to \check{R}$.  The possibilities are
\[
(5,1) \;,\, (4,1,1) \;,\, (3,3) \;,\, (3,1,1,1) \;,\, (2,2,2) \;,\, (2,2,1,1) \;,\, (2,1,1,1,1).
\]
If the generic splitting type is $(2,2,2)$, then the point $p$ is the cone vertex $w$ and in this case $k_p = 0$.  The partition $(2,1,1,1,1)$ corresponds to the case $d_i d_s = 1$ and we already analyzed it.  The partitions $(5,1)$ and $(4,1,1)$ can also be excluded, since they imply that $R$ is birational to a plane curve of degree at most $\lfloor \frac{18}{4} \rfloor = 4$ and hence cannot have arithmetic genus four.

Examining the remaining generic splitting types and using Lemma~\ref{uti}, we deduce that the $p$-regular tritangent planes correspond to ramification points of $g _s$.  In particular there are at most $42+28=70$ such points, and the proposition follows. 
\end{proof}

\begin{remark} \label{c2gen}
In characteristic two the curve $R$ above has degree $3$, contains the cone vertex $w$, and is
not necessarily smooth (see~\cite{CO}).  While we are not able to
prove that every del Pezzo surface of degree one in characteristic two is {\free}, we prove it for a general such
surface.  To show this note that the property of being {\free} is
open; thus it suffices to exhibit a single example of a del Pezzo
surface $X$ of degree one defined over $\overline{\mathbb{F}}_2$ and five
tritangent planes to $X$ no four of which share a common point.  We found this example using {\tt Magma}~\cite{Magma}.

Let $\mathbb{P} := \mathbb{P}(1,1,2,3)$ be the weighted projective space over $\overline{\mathbb{F}}_2$
with coordinates $x,y,z,w$ and respective weights $1,1,2,3$.  Let $X
\subset \mathbb{P}$ be the surface defined by
\[
F \colon \bigr\{ w^2 + z^3 + wx z + wy^3 + x^6 = 0 \bigr\}.
\]
It is immediate to check that $X$ is smooth.  Let $t \in \overline{\mathbb{F}}_2$ be an
element satisfying ${t^2+t+1=0}$; any four of the homogeneous forms
\[
\begin{array}{l@{\hspace{25pt}}l}
z + x^2 , & z + t x^2 + t^2 xy + y^2, \\
z + y^2 , & z + t x^2 + xy + t^2 y^2 , \\
z + t x^2
\end{array}
\]
are linearly independent.  Moreover, they represent tritangent
planes since we have
\[
\begin{array}{rcl}
F(x,y,x^2,w) & = & w (w+x^3+y^3) \\
F(x,y,y^2,w) & = & \bigl( w+(x^3+txy^2+x^2y+ty^3) \bigr) \\
             &   &\quad\cdot\bigl( w+(x^3+t^2xy^2+x^2y+t^2y^3) \bigr) \\
F(x,y,tx^2,w) & = & w (w+tx^3+y^3) \\
F(x,y,t x^2 + t^2 xy + y^2,w) & = & \bigl( w + y (x^2 + xy + ty^2)
\bigr) \bigl( w + (tx^3 + t^2y^3 + tx^2y) \bigr) \\
F(x,y,t x^2 + xy + t^2 y^2,w) & = & \bigl( w + ty (x^2 + txy + y^2)
\bigr) \bigl( w + t (x^3 + tx^2y + ty^3) \bigr)
\end{array}
\]
and we deduce that $X$ is {\free}.
\end{remark}

The proof of Corollary~\ref{Popovgenericity} below given in~\cite{BP}*{Proposition~3.4} contains a gap pointed out in~\cite{Popov}: the original 
argument implicitly assumes that the characteristic of the base field is not two, and reduces the proof to the fact that $\HX{0}{-2K_X}$ is spanned by 
the sections supported on exceptional curves.  Proposition~\ref{tritangenti} fixes this gap.

\begin{cor} \label{Popovgenericity}
Let $X$ be a del Pezzo surface of degree one in characteristic not two. If $D\neq-K_X$ is an effective divisor, then $\HX{0}{D}$ is spanned by global sections supported on exceptional curves.  \qed
\end{cor}

\begin{remark} \label{PoC2}
A general del Pezzo surface of degree one in characteristic two is {\free}; therefore Corollary~\ref{Popovgenericity} also holds for such surfaces.
\end{remark}

\section{Betti numbers and the Batyrev--Popov conjecture}
\label{S:betti}

We review the approach of~\cite{LV} to study the ideal of relations of Cox rings. As usual, let $X$ be a del Pezzo surface, and let $R = k[\cG]$. Since $R$ is positively graded (by 
anticanonical degree), every  finitely generated $\Pic(X)$-graded $R$-module has a unique minimal 
$\Pic(X)$-graded free resolution. For the module $\Cox(X) = R/I_X$ this resolution 
is of the form
\[
\dots \rightarrow \bigoplus_{D\in \Pic(X)} R(-D)^{b_{2,D}}\rightarrow \bigoplus_{D\in \Pic(X)} R(-D)^{b_{1,D}} \rightarrow R \rightarrow 0,
\]
where the rightmost non-zero map is given by a row matrix whose non-zero entries are a set of minimal generators of the ideal $I_X$. Since the differential of the resolution has degree $0$, it follows that $I_X$ has exactly $b_{1,D}(\Cox(X))$ minimal generators of Picard degree $D$.

Let $\mathbb{K}$ be the Koszul complex on $\cG$. Consider the degree $D$ 
part of the complex $\Cox(X)\otimes_R \mathbb{K}$. Then
\[
{\rm H}_i \big( (\Cox(X)\otimes_R \mathbb{K})_D \big) = \big( {\rm H}_i(\Cox(X)\otimes_R \mathbb{K}) \big)_D =\big( \Tor^R_i (\Cox(X),k) \big)_D = k^{b_{i,D}(\Cox(X))},
\]
where the last two equalities follow since $\Tor_i^R(A,B)$ is symmetric in $A$ and $B$ and the Koszul complex is the minimal free resolution of $k$ over $R$.  Hence we have the equality
\[
\dim_k \big( {\rm H}_i\big( \bigl(\Cox(X) \otimes_R \mathbb{K}\bigr)_D \big) \big) = b_{i,D}(\Cox(X)).
\]
Thus, Conjecture~\ref{C:BP} is equivalent to the statement that $b_{1,D}(\rm{Cox}(X))=0$ for all $D\in \Pic(X)$ with $-K_X\cdot D\geq 3$. This is the form of the conjecture that we prove.

Let $X$ be a del Pezzo surface of degree one.  
To compute the Betti numbers $b_{1,D}(\Cox(X))$, denote by $C_1 , \ldots , C_{240}$ the exceptional curves of $X$, and let $C_{240 + i} := K_i$ for $i \in \{1, 2\}$. With this notation, the part of the complex relevant to our task is
\begin{equation*}
\bigoplus_{1\leq i<j\leq 242} {\rm H}^0 \big( X, \cO_X(D-C_i-C_j) \big) \xrightarrow{d_2} \bigoplus_{i=1}^{242} {\rm H}^0 \big( X, \cO_X(D-C_i) \big) \xrightarrow{d_1} {\rm H}^0 \big( X, \cO_X(D)\big),
\end{equation*}
where $d_2$ sends $\sigma_{ij} \in {\rm H}^0 \big( X, \cO_X(D-C_i-C_j) \big)$ to $(0,\dots,0,\sigma_{ij}c_j,0,\dots,0,-\sigma_{ij}c_i,0,\dots,0)$ and $d_1$ sends $\sigma_i \in {\rm H}^0 \big( X, \cO_X(D-C_i) \big)$ to $\sigma_ic_i$.

A \defi{cycle} is an element of $\ker d_1$; a \defi{boundary} is an element of ${\rm im}\, d_2$. The \defi{support} of a cycle $\sigma = ( \ldots , \sigma_i , \ldots )$ 
is 
\[
|| \sigma ||= \big\{ C_i :  \sigma_i\neq 0 \big\}
\]
and the size of the support is the cardinality of $|| \sigma ||$, denoted  $| \sigma |$.

\subsection{Strategy}
In order to show that $b_{1,D}(\Cox(X)) = 0$ whenever $-K_X\cdot D \geq 3$, we split the divisor classes $D$ as follows:
\begin{enumerate}
\item $D$ is ample and has anticanonical degree at least four (\S\ref{S:ample four});
\item $D$ is ample and has anticanonical degree three (\S\ref{S:ample three});
\item $D$ is not ample (\S\ref{S:non ample divisors}).
\end{enumerate}

The case when $D$ is not ample follows by induction on the degree of the del Pezzo surface, and the ample divisors of anticanonical degree three are dealt with algebraically. In order to show that $b_{1,D}(\Cox(X)) = 0$ whenever $D$ is ample and $-K_X \cdot D \geq 4$ we show that every cycle is a boundary by arguing as follows.
\begin{enumerate}
\item \label{qui1}
We describe ways in which a divisor may be removed from the support of a cycle using boundaries, at the cost of possibly introducing new divisors in the support of the cycle.
\item We apply the constructions in (\ref{qui1}) to all cycles in a systematic way to reduce their support 
to at most two elements.
\end{enumerate}

We conclude using the following lemma.

\begin{lemma} 
\label{lem:easy} 
Any cycle $\sigma$ with $|\sigma|\leq 2$ is a boundary.
\end{lemma}

\begin{proof}
We deduce from~\cite{EKW}*{Theorem~1.1} that the Cox rings of del Pezzo surfaces are unique factorization domains. The statement follows immediately.
\end{proof}

\section{Capturability and Stopping criteria}
\label{S:cap and stop}

We introduce the following terminology to aide us in our search of ways to reduce the support of a cycle. We keep the notation of the previous section.

\begin{definition}
A \defi{capture move for a divisor $D$} is a pair $(\cS , C )$, where $\cS \subset \cG$, $C \in \cG \setminus \cS$ 
and the map
\[
\bigoplus_{S \in \cS} {\rm H}^0 \big( X, \cO_X(D - S - C) \big) \otimes {\rm H}^0\big ( X, \cO_X(S) \big) \longrightarrow {\rm H}^0 \big( X, \cO_X(D - C) \big)
\] 
induced by tensor product of sections is surjective. We say that $C$ is the \defi{captured curve}, and that $C$ is \defi{capturable for $D$ by} $\cS$.
\end{definition}

Let $T$ be a set; for a direct sum $\bigoplus_{t \in T} V_t$ of vector spaces indexed by $T$ and $T'\subset T$, we write $\sum_{t \in T'} a_t \hat{e}_t$ for the element $(b_t)_{t\in T}$ of $\bigoplus_{t \in T} V_t$ where $b_t=a_t$ if $t\in T'$ and $b_t=0$ otherwise.  As usual, if $S \in \cG$, we denote by $s$ the chosen global section of $\mathcal{O}_X(S)$.   Let $( \cS , C )$ be a capture move for $D$, let $\sigma_C \in {\rm H}^0 \big( X, \cO_X(D - C) \big)$, and let $\sigma \in \bigoplus {\rm H}^0 \big( X, \cO_X(D - C_i) \big)$ be the element corresponding to $\sigma _C$.  Then we have
\[
\sigma_ C = \sum_{S \in \cS} p_s s, \qquad p_s \in {\rm H}^0 \big( X, \cO_X(D - S - C) \big),
\]
and thus we obtain
\[
\sigma = \sum_{S \in \cS} p_s c\, \hat{e}_s  + d_2 \Bigl( \sum_{S \in \cS} \varepsilon _{CS} p_s \hat{e}_{cs} \Bigr) \in \bigoplus_{C_i}  {\rm H}^0 \big( X, \cO_X(D - C_i) \big)
\]
where $\varepsilon _{CS} \in \{ \pm 1 \}$.  

Hence, if $\sigma$ is a cycle and there is a capture move $( \cS , C )$, then we can modify 
$\sigma$ by a boundary so that $C \notin ||\sigma ||$. In this sense we have {\it captured} the 
curve $C$ from the support of $\sigma $.  Note, however, that we may have 
added $S$ to $|| \sigma ||$, for all $S \in \cS$, so \emph{a priori} the size of the support may 
not have decreased. We need to find and apply capture moves in an organized way to 
ensure that we are genuinely decreasing the size of the support of a cycle.

Finally note that if $(\cS , C )$ is a capture move and $\cS' \supset \cS$ with $C \notin \cS'$, 
then also $(\cS' , C )$ is a capture move.  We frequently use this observation without 
explicitly mentioning it.

\begin{lemma}
\label{captureC}
Let $A$, $B$ and $C$ be exceptional curves on $X$, $A$ and $B$ disjoint, and let $D$ be any divisor.  
If $\HX{1}{D-A-B-C}=0$, then $(\{A,B\},C)$ is a capture move for $D$. In particular, if $({-K}_X+D-A-B-C)$ is nef, and if either it is big or it has anticanonical degree two, then $(\{A,B\},C)$ is a capture move for $D$.
\end{lemma}

\begin{proof}
Since $A$ and $B$ are disjoint, there is an exact sequence of sheaves
\[
0\rightarrow \cO_X(-A-B)\rightarrow \cO_X(-A)\oplus\cO_X(-B)\rightarrow O_X\rightarrow 0 .
\]
Tensoring with $\cO_X(D-C)$ we obtain the short exact sequence
\[
0\rightarrow \cO_X(D-A-B-C)\rightarrow \cO_X(D-A-C)\oplus\cO_X(D-B-C)\rightarrow O_X(D-C)\rightarrow 0
\]
and the desired surjectivity follows from the associated long exact sequence in cohomology and the assumption that $\HX{1}{D-A-B-C}=0$. The last statement follows from the Kawamata-Viehweg vanishing theorem when ${-K}_X+D-A-B-C$ is nef and big. Otherwise ${-K}_X+D-A-B-C = Q$ for some conic $Q$ and the statement follows from the exact sequence 
\[
0\rightarrow \cO_X(K_X)\rightarrow \cO_X(K_X+Q)\rightarrow \cO_{Q}(K_X+Q)\rightarrow 0,
\]
by considering the associated long exact sequence in cohomology.
\end{proof}

\begin{lemma}
\label{captureK}
If $X$ is {\free}, then the pairs $(\cC,K_1)$ and $(\cC,K_2)$ are capture moves for all $D\neq -2K_X$.
\end{lemma}

\begin{proof}
Since $D+K_X \neq -K_X$, Corollary~\ref{Popovgenericity} and Remark~\ref{PoC2} show that $\HX{0}{D+K_X}$ 
is spanned by global sections supported on exceptional curves. Thus the  following map 
is surjective
\[
\bigoplus_{S \in \cC} {\rm H}^0 \big( X, \cO_X(D - S +K_X) \big) \otimes {\rm H}^0\big ( X, \cO_X(S) \big) \longrightarrow {\rm H}^0 \big( X, \cO_X(D + K_X) \big)
\] 
and $(\cC, K_1)$ and $(\cC,K_2)$ are capture moves for $D$.
\end{proof}

\begin{definition}
A \defi{stopping criterion for a divisor $D$} is a set $\cS \subset \cG$ such that the complex
\[
\bigoplus_{\substack{C_i, C_j \in \cS \\ i < j}} {\rm H}^0 \big( X, \cO_X(D-C_i-C_j) \big) \xrightarrow{d_2} \bigoplus_{C_i \in \cS} {\rm H}^0 \big( X, \cO_X(D-C_i) \big) \xrightarrow{d_1} {\rm H}^0 \big( X, \cO_X(D)\big)
\]
is exact.
\end{definition}

\begin{remark}
By Lemma~\ref{lem:easy}, a subset of $\cG$ of cardinality two is a stopping criterion for any divisor $D$.
\end{remark}

The name stopping criterion is motivated by Theorem~\ref{thm:games}: whenever we can capture all curves in a given degree $D$ by curves contained in a stopping criterion, 
then there are no relations in degree $D$.  In this case, we may {\it stop} looking for capture moves.

\begin{definition}
Let $D \in \Pic(X)$ be a divisor class and let 
\[
\cM := \big( E_1 , \dots , E_n \big)
\]
be a sequence of elements of $\cG$; define $\cS_i := \cG \setminus \{ E_1 , E_2 , \ldots , E_i \}$ 
for all $i \in \{ 1 , \ldots , n \}$.  We say that $D$ is \defi{capturable} (by $\cM$) if 
\begin{enumerate}
\item $( \cS_i , E_i )$ is a capture move for $D$ for all $i \in \{ 1 , \ldots , n \}$, 
\item $\cS_n$ is a stopping criterion for $D$.
\end{enumerate}
\end{definition}

\begin{theorem}
\label{thm:games} 
Fix a divisor class $D \in \Pic(X)$. If $D$ is capturable then 
\[
b_{1,D}(\Cox(X))=0
\]
and hence there are no minimal generators of the ideal $I_X$ in degree $D$.
\end{theorem}

\begin{proof}  
Let $\cM := ( E_1, \dots, E_n )$ be such that  $D$ is capturable by $\cM$. Let $\cS_0=\cG$ and, 
for $i \in \{1 , \ldots , n \}$ define $\cS_i := \cG \setminus \{ E_1 , E_2 , \ldots , E_i \}$.  
We want to show that every cycle $\sigma$ is a boundary.  Let 
$j (\sigma ) := \max \{ i \in \{ 0 , \ldots , n \} \,|\, || \sigma || \subset \cS_i \}$.  
By definition if $|| \sigma || \subset \cS _n$, or equivalently if $j (\sigma ) = n$, then 
$\sigma $ is a boundary.

Suppose that $j = j (\sigma ) < n$; since $( \cS_{j+1} , E_{j+1} )$ is a capture move for $D$, we may 
modify the support of $\sigma$ by a boundary (if necessary) to ensure that 
$||\sigma || \subset \cS_{j+1} = \cS_j \setminus \{ E_{j+1} \}$.  Repeating this argument, 
we deduce that, modulo boundaries, we may assume that the support of $\sigma $ is 
contained in $\cS_n$ and thus that $\sigma $ is itself a boundary.
\end{proof}

\section{Ample divisors of anticanonical degree at least four}
\label{S:ample four}

Let $X$ be a del Pezzo surface of degree one.  Assume throughout this section that if $X$ is defined over a field of characteristic two, then $X$ is {\free}.  In this section we prove that there are no minimal generators of the ideal $I_X$ in all sufficiently large ample degrees. More precisely, if $D$ is an ample divisor such that $-K_X\cdot D \geq 4$ then we show that $D$ is capturable. 

The general strategy is the following. First we capture $K_1$ and $K_2$ using $\cC$ via Lemma~\ref{captureK}. Next, assume $C$ is an exceptional curve and we want to capture it using $\cS \subset \cC$. We prove that there exist disjoint exceptional curves $S$ and $T$ in $\cS$ such that $\HX{1}{K_X + L_{ST}} = 0$, where $L_{ST} := -K_X + D - C - S - T$. Then $(\{S,T\},C)$ is a capture move for $D$ by Lemma~\ref{captureC}. Often we show that $L_{ST}$ is nef and either big or has anticanonical degree two, omitting any reference to Lemma~\ref{captureC}.  In all cases we capture enough curves to conclude that $D$ is a capturable divisor, using Lemma~\ref{lem:easy} as stopping criterion.

\begin{lemma} 
\label{ma7}
Let $X$ be a del Pezzo surface of degree one and let $b \colon X \to Y$ be a birational 
morphism.  Denote by $K_Y$ the pull-back to $X$ of the canonical divisor on $Y$ and 
let $N$ be a nef divisor on $X$.  Assume that $\deg (Y) \geq 3$; then the divisor 
$D = -K_X-K_Y+N$ is capturable.
\end{lemma}

\begin{proof}
Let $\cS$ be the set of exceptional curves contracted by $b$ and note that $\cS$ consists 
of at least two disjoint exceptional curves and that $-K_Y - S - T$ is big and nef for all 
$S,T \in \cS$, $S \neq T$.  Let $C \in \cC \setminus \cS$ be any exceptional curve and let 
$C' := -2K_X-C$.  First we capture the curves $C$ such that $C' \in \cS$ using $\cS$: 
$L_{C'T}:= -K_Y+N-T$ is big and nef for any choice of $T \in \cS \setminus \{C'\}$.  

Second we choose any two distinct $S,T \in \cS$, we let $\cS' := \{S,T\}$ and we capture 
all curves $C \in \cC \setminus \cS'$ such that $C' \notin \cS$: the divisor 
$L_{ST} = C'+(-K_Y-S-T)+N$ is big and nef provided $C' \cdot (-K_Y-S-T) > 0$; since the 
only curves orthogonal to $(-K_Y-S-T)$ are the curves in $\cS \setminus \cS'$, we conclude.
\end{proof}

\begin{lemma} 
\label{ma8}
Let $X$ be a del Pezzo surface of degree one and let $b \colon X \to Y$ be a birational 
morphism.  Denote by $A$ the pull-back to $X$ of the minimal ample divisor on $Y$ and 
let $N$ be a nef divisor on $X$.  Assume that $\deg (Y) \geq 8$; then the divisor 
$D = -K_X+A+N$ is capturable.
\end{lemma}

\begin{proof}
Note that $Y$ is isomorphic to $\bP^2$, $Bl_p(\bP^2)$ or $\bP^1 \times \bP^1$.

{\bf Case 1: $Y \not \simeq \bP^1 \times \bP^1$.}

\noindent
In this case $A=L+Q$, where $L$ is a twisted cubic and $Q$ is either zero or a conic.  
Thus it suffices to treat the case $D = -K_X+L+N$.

Let $\cS$ be the set of curves contracted by $L$; thus $\cS$ consists of eight 
disjoint exceptional curves.  Note that $\sum _{S \neq T \in \cS} L-S-T = 7(-K_X+L)$
has intersection number at least seven with every exceptional curve on $X$.  Since 
the intersection number of two exceptional curves on $X$ is at most three, it follows 
that every exceptional curve on $X$ intersects positively at least two of the curves 
$\{L-S-T \,|\, S \neq T \in \cS\}$.  For any $C \in \cC \setminus \cS$ let $S,T \in \cS$ be 
such that $(-2K_X-C) \cdot (L-S-T) > 0$ and such that no (rational) multiple of $N$ 
equals $(-2K_X-C) + (L-S-T)$.  Note that the second condition ensures that $N$ does not contract both $-2K_X-C$ and $(L-S-T)$ when $-2K_X-C+(L-S-T)$ is a conic. 
We have that $L_{ST}:= -2K_X+L-C-S-T+N = (-2K_X-C)+(L-S-T)+N$ is nef and either it has 
anticanonical degree two or it is also big.  Thus $(\cS,C)$ is a capture move for all 
$C \in \cC \setminus \cS$.  
Let $S,T \in \cS$ be distinct elements and let $\cS' := \{S,T\}$.  
For any $C \in \cS \setminus \cS'$ the divisor $L_{ST}:= (-2K_X-C)+(L-S-T)+N$ is big and 
nef since $(-2K_X-C) \cdot (L-S-T) = 2$.

{\bf Case 2: $Y \simeq \bP^1 \times \bP^1$.}

\noindent
Let $\cS$ be the set of curves contracted by $A$; thus $\cS$ consists of seven disjoint 
exceptional curves.  Note that $\sum _{S \neq T \in \cS} A-S-T = -6K_X+9A$ has 
intersection number at least six with every exceptional curve on $X$ and that the 
summands are conics.  Since the intersection number of an exceptional curve and a 
conic on $X$ is at most four, it follows that every exceptional curve on $X$ intersects 
positively at least two of the conics $\{A-S-T \,|\, S \neq T \in \cS\}$.  For any 
$C \in \cC \setminus \cS$ let $S,T \in \cS$ be such that $(-2K_X-C) \cdot (A-S-T) > 0$. 
We have that $L_{ST}:= -2K_X+A-C-S-T+N = (-2K_X-C)+(A-S-T)+N$ is big and nef.  Thus $(\cS,C)$ is 
a capture move for all $C \in \cC \setminus \cS$.  Let $S,T \in \cS$ be distinct elements 
and let $\cS' := \{S,T\}$.  For any $C \in \cS \setminus \cS'$ the divisor 
$L_{ST}:= (-2K_X-C)+(A-S-T)+N$ is big and nef since $(-2K_X-C) \cdot (A-S-T) = 4$.
\end{proof}

\begin{lemma} \label{ke}
Let $X$ be a del Pezzo surface of degree one, $E$ an exceptional curve on $X$ and 
$N \neq 0$ a nef divisor on $X$ such that $N \cdot E = 0$.  The divisor $-(n+2)K_X+E+N$ is 
capturable for $n \geq 0$.
\end{lemma}

\begin{proof} By Lemma~\ref{mad}, for $i \in \{1,2\}$, there is an exceptional curve $F_i$ on 
$X$ such that $F_i \cdot E = 0$ and $N-F_i$ is either nef or the sum of a nef divisor and 
an exceptional curve, with $F_1 \neq F_2$, .  
Let $\cS := \{E,F_1,F_2\}$.  For any $C \in \cC \setminus \cS$, let $F \in \{F_1,F_2\}$ be 
such that $(-2K_X-C)$ is not a fixed component of $(N-F)$; we have that 
$L_{EF} = -(n+3)K_X + E - C - E - F + N = -K_X+(-2K_X-C) + (N-F) -n K_X$ is big and 
nef and we conclude that we can capture all curves in $\cC \setminus \cS$, using only 
the three curves in $\cS$.  Finally, by the same argument, we find $(\{E,F_1\} , F_2)$ 
is a capture move and we are done.
\end{proof}

\begin{lemma} \label{kke}
Let $X$ be a del Pezzo surface of degree one, $E$ an exceptional curve on $X$ and 
$N$ a nef divisor on $X$.  The divisor $D:=-(n+3)K_X+E+N$ is capturable if $n \geq 0$.
\end{lemma}

\begin{proof}
Let $E , E_2 , \ldots , E_8$ be eight disjoint exceptional curves on $X$ and let $\tilde E_i:= -K_X+E-E_i$ for $2\leq i\leq 8$. Let $\cS := \{E , E_2, \ldots , E_8 , \tilde E_2, \ldots , 
\tilde E_8\}$.  It is clear from Table~\ref{tab:-1curves1} that every exceptional curve $C \in \cC \setminus \cS$ intersects positively at least two of the curves 
in $\cS$; in particular for every $C \in \cC \setminus \cS$ there is $S \in \cS$ such that 
$S \cdot E = 0$, $S \cdot C > 0$ and $N$ is not a (rational) multiple of $-4K_X-(S+C)$.  
Thus $L_{ES} = -K_X+D-E-S-C=(-4K_X-C-S)+N-nK_X$ is nef and it either 
has anticanonical degree two or it is big.  In either case $\HX{1}{K_X+L_{ES}}=0$ and 
we may capture all the curves in $\cC \setminus \cS$ using the curves in $\cS$.  

Note that for all integers $i,j,k,l \in \{ 2 , 3 , \ldots , 8 \}$, $i \neq j$, $k \neq l$, we have 
$(-4K_X-E_i-\tilde E_j) \cdot (-4K_X-E_k-\tilde E_l) = 0$ if and only if $i=k$ and $j=l$.  
In particular, if $N$ is a multiple of a conic, it is proportional to at most one of the divisors 
$(-4K_X-E_i-\tilde E_j)$; relabeling the indices if necessary, we may assume that $N$ is 
not proportional to $-4K_X-(E_i+\tilde E_j)$ for all $i,j \in \{2 , 3 , \ldots , 8 \}$ and 
$i \neq j$, $(i,j) \neq (7,8)$.  Thus for all $i \in \{3 , 4 , \ldots , 8 \}$ the divisor 
$L_{E \tilde E_2} = (-4K_X-E_i-\tilde E_2)+N-nK_X$ is either a conic or big and nef and 
we may capture $E_3 , E_4 , \ldots , E_8$ using $E,\tilde E_2$.  Similarly we can capture 
$\tilde E_2 , \tilde E_3 , \ldots , \tilde E_8$ using $E,E_2$ and we are done.
\end{proof}

\begin{lemma} 
\label{lem:capture} 
Let $X$ be a del Pezzo surface of degree one. Let $D=-(n+4)K_X+N$ where $n \geq 0$ and $N$ is a nef divisor. Let $C,S$ and $T$ be exceptional curves such that $S\cdot T=0$ and that $C\cdot S$, $C\cdot T$ are at least $2$. Then $(\{S,T\},C)$ is a capture move for $D$ if either
\begin{enumerate}
\item{$C\cdot (S+T)=5$, or }
\item{$C\cdot(S+T)=4$ and the divisor $N$ is either $0$ or not a multiple of the conic $T'+S'+C'+K_X$.}
\end{enumerate}
\end{lemma}
\begin{proof} Consider $L_{ST}:=-K_X+D-S-T-C=S'+T'+C'-(n-1)K_X+N$. Since $C'\cdot S'$ and $C'\cdot T'$ are at least two the curves $S',C'$ and $T'$, taken together, are not all pullbacks of exceptional curves on a del Pezzo surface $Y$ of degree at least two. It follows that every exceptional curve $V\not\in\{S',T',C'\}$ intersects at least one curve in this set and thus $L_{ST}$ is nef. If $C\cdot(S+T)=5$ then either $C' + S' = -2K_X$ or $C' + T' = -2K_X$ and  $L_{ST}$ is big. If $N=0$ then $L_{ST}$ has anticanonical degree $2$ and capturability follows from Lemma~\ref{captureC}.
Finally, assume that $C\cdot S=C\cdot T=2$. In this case $S'+T'+C'+K_X=Q$ is a conic and $L_{ST}=Q-nK_X+N$ is big unless $N=mQ$ for some $m\geq 1$.
\end{proof}

\begin{lemma} 
\label{kk}
Let $X$ be a del Pezzo surface of degree one and let $N$ be a nef divisor on $X$.  
The divisor $-(n+4)K_X+N$ is capturable for $n \geq 0$.
\end{lemma}

\begin{proof}
Let $L$ be a twisted cubic and let $E_1,\dots, E_8$ be exceptional curves on $X$ contracted by $L$. First we capture all curves $C\in\cC$ such that $L\cdot C=3$. From Table~\ref{tab:-1curves1} it is clear that $=K_X+E_i-E_j$ for some $i\neq j$. Let $S=E_j$, $T_1=2L-E_i-E_{i_1}-\dots E_{i_4}$ where $\{i,j\}\cap \{i_1,\dots, i_4\}=\emptyset$ and $T_2=2L-E_i-E_{j_1}-\dots E_{j_4}$ where $\{i,j\}\cap \{j_1,\dots, j_4\}=\emptyset$ and furthermore $\{i_1,\dots, i_4\}\neq\{j_1,\dots, j_4\}$.
Note that $L\cdot S$, $L\cdot T_i\neq 3$, $C\cdot S=C\cdot T_i=2$ and $S\cdot T_i=0$ for $i=1,2$. Let $T$ be an exceptional curve in $\{T_1,T_2\}$ such that $N\neq m(T'+S'+C'+K_X)$. By Lemma~\ref{lem:capture}(ii), $(\{S,T\},C)$ is a capture move.   

Next, we capture curves $C$ with $L\cdot C\in\{4,5\}$. It is clear from Table~\ref{tab:-1curves1} that for any such curve there exists a pair of distinct exceptional curves $S,T\in \{E_1,\dots, E_8\}$ such that $(\{S,T\},C)$ is a capture move via Lemma~\ref{lem:capture}(ii). Similarly we capture curves $C$ with $L\cdot C\in\{1,2\}$ using some pair $S,T\in \{E_1',\dots, E_8'\}$. Next we capture $E_3 , \ldots , E_8$ using $E_1'$ and $E_2'$ via Lemma~\ref{lem:capture}(i) and finally, we capture $E_1' , \ldots , E_8'$ using $E_1$ and $E_2$.  This concludes the proof.
\end{proof}

\begin{lemma} 
\label{cocco}
Let $X$ be a del Pezzo surface of degree one and let $Q$ be a conic on $X$.  
The divisor $D=-(n+1)K_X+(m+1)Q$ is capturable for $m,n \geq 0$ and $m+n \geq 1$.
\end{lemma}

\begin{proof}
Suppose first that $n \geq 1$. Let $\cS$ be the set of curves contracted by $Q$ and for any $S\in \cS$, let $\tilde{S}$ be the unique curve such that $S+\tilde{S}=Q$. If $S,T,C$ are exceptional curves let $B_{ST}=-K_X+C'+Q-S-T$ and note that if $B_{ST}$ is nef and big then so is $L_{ST}:=-K_X+D-C-S-T=B_{ST}+mQ-(n-1)K_X$ because $n\geq 1$. Moreover $B_{ST}$ has anticanonical degree $2$ and $\hX{2}{B_{ST}}=0$ so by Riemann-Roch $B_{ST}$ is nef if and only if $B_{ST}^2=2Q\cdot C'-2C'\cdot S-2C'\cdot T\geq 0$. 

We first show that the curves $C$ in $\cC\setminus \cS$ can be captured with the curves in $\cS$. To do so, we split them into cases according to the value of $Q\cdot C'$.

If $Q\cdot C'=0$ choose $T\in \cS$ disjoint from $C'$. In this case $B_{C'T}=-K_X+\tilde{T}$ is nef and big, and $(\{C',T\},C)$ is a capture move. If $Q\cdot C'=1$ choose disjoint curves $S$ and $T\in \cS$ which are also disjoint from $C'$ and note that $B_{ST}$ has anticanonical degree $2$ and square $2$ so it is the pullback to $X$ of the anticanonical divisor of a del Pezzo surface of degree $2$ and thus it is nef and big; whence $(\{S,T\},C)$ is a capture move. If $Q\cdot C'=2$, by Lemma~\ref{dose} there exist disjoint curves $S$,$T$ in $\cS$ such that $S\cdot C'=T\cdot C'=0$ so $B_{ST}$ has square $4$ and hence $B_{ST}=-2K$ is nef and big. Thus $(\{S,T\},C)$ is a capture move in this case. Finally if $Q\cdot C'=3$, then let $S$ and $T$ be disjoint curves in $\cS$ such that $S\cdot C'=T\cdot C'=1$; then $B_{ST}$ has square $2$ and it is a nef and big divisor since it is the pullback to $X$ of the anticanonical divisor of a del Pezzo surface of degree $2$. Thus $(\{S,T\},C)$ is a capture move in this case. 

Now choose $S$, $T$ disjoint exceptional curves in $\cS$. We show that all curves $C$ in $\cS\setminus\{S,T\}$ can be captured from $\{S,T\}$. In this case $Q\cdot C'=4$ and $S\cdot C'$, $T\cdot C'\leq 2$, so $B_{ST}^2\geq 0$. If $B_{ST}$ is big the statement follows since so is $L_{ST}$. If $B_{ST}^2=0$ there are two cases to consider, either $n=1$ and $m=0$ and the statement follows from Lemma~\ref{captureC}, or at least one of $n$, $m$ is greater than one. In this case $L_{ST}=B_{ST}-(n-1)K_X+mQ$ is big because $B_{ST}\cdot Q=6\neq 0$. It follows that we can capture every exceptional curve in $\cS\setminus \{S,T\}$ from $\{S,T\}$.

Suppose now that $n=0$ and $m \geq 1$.  Note that if $Q\cdot C'\leq 1$ then $Q\cdot (D-C)=Q\cdot(K_X+C'+(m+1)Q)=-2+Q\cdot C'<0$ so $D-C$ is not an effective divisor. In this case $C$ is captured vacuously and thus we restrict our attention to curves $C$ with $Q\cdot C'\geq 2$.
Let $\cS$ be the set of exceptional curves contracted by $Q$ and for any $S\in \cS$ let $\tilde{S}$ be the unique curve such that $S+\tilde{S}=Q$. 

For any $C\in\cC\setminus \cS$ which satisfies $Q\cdot C'\geq 2$ there exist disjoint curves $S,T$ in $\cS$ such that $\tilde{S}$ and $\tilde{T}$ intersect $C'$. By construction $\tilde{S}+\tilde{T}+C'$ is nef and big and so is $L_{ST}:=-K_X+D-S-T-C=\tilde{S}+\tilde{T}+C'+(m-1)Q$. It follows that we can capture $\cC\setminus \cS$ with $\cS$.

Finally, fix two disjoint curves $S,T\in\cS$ and let $C$ be any curve in $\cS\setminus\{S,T\}$. Since $Q\cdot C=0$ we have $Q\cdot C'=4$ so $C'$ intersects both $\tilde{S}$ and $\tilde{T}$. It follows that $L_{ST}$ is nef and big and that $\cS$ can be captured from $\{S,T\}$.
\end{proof}

\begin{theorem} \label{non4}
Let $X$ be a sweeping del Pezzo surface of degree one and let $D$ be an ample divisor.  If $-K_X \cdot D \geq 4$, then $D$ is capturable.
\end{theorem}

\begin{proof}
Write $D = -nK_X+N$, where $n \geq 1$, the divisor $N$ is nef and not ample and 
$n-K_X \cdot N \geq 4$.  If $N=0$, then $n \geq 4$ and we conclude using 
Lemma~\ref{kk}.  If $N \neq 0$, write $N = mA+N'$ where $X \to Y$ is a morphism 
with connected fibers, $A$ is the pull-back to $X$ of the minimal ample divisor on 
$Y$, $N'$ is the pull-back of a nef divisor on $Y$ and $m \geq 1$.  If $Y$ 
is a surface, then we conclude using Lemmas~\ref{ma7},~\ref{ma8},~\ref{ke} and~\ref{kke}.  If $Y \simeq \bP^1$, then we conclude using 
Lemma~\ref{cocco}.
\end{proof}

\section{Ample divisors of anticanonical degree three}
\label{S:ample three}

The only ample divisors $D$ that seem to elude the strategy of \S\ref{S:ample four} are those 
of anticanonical degree three.  These divisors are $-3K_X$, $-2K_X+E$, $-K_X+Q$, 
where $E$ is any exceptional curve and $Q$ is any conic.  Recall that there is a surjection
\[
(k[\cG]/J_X)_D\to (k[\cG]/I_X)_D;
\]
see \S\ref{S:notation}.  We show that the dimension of $(k[\cG]/J_X)_D$ is at most $\hX{0}{D}=\dim(k[\cG]/I_X)$ and conclude that $(I_X)_D=(J_X)_D$. 

\begin{lemma} 
\label{basis-K+G}
Let $X$ be a del Pezzo surface of degree one and let  $A,B,G$ be exceptional curves with $A+B=-K_X+G$. Then $k_1g,k_2g$ and $ab$ form a basis for $\HX{0}{-K_X+G}$. In particular, if $C,D$ are exceptional curves with $C+D=-K_X+G$ then $cd=\alpha_1k_1g+\alpha_2 k_2g+\alpha_3ab$ and the coefficient $\alpha_3$ is non-zero.
\end{lemma}

\begin{proof}
The morphism $X \to \mathbb{P}^2$ associated to $|{-K}_X+G|$ contracts $G$ and is ramified along a smooth plane quartic $R$. The image of $G$ is a point not lying on any bitangent to $R$ since otherwise $X$ would have had a $(-2)$-curve. The images of $K_1$ and $K_2$ in $\mathbb{P}^2$ are distinct lines through the image of $G$ and therefore the three lines in $\mathbb{P}^2$ corresponding to $A+B$, $K_1+G$ $K_2+G$ have no common point and are independent.
\end{proof}

\begin{lemma} 
\label{tretre}
Let $X$ be a del Pezzo surface of degree one; then the ideal $I_X$ has no minimal generators in degree $D=-3K_X$.
\end{lemma}

\begin{proof}
The only ways of writing $-3K_X$ as a sum of three effective divisors are given in the following table.
\begin{equation} \label{che}
\begin{array}{|l|l|}
\hline 
{\text{Monomial}} & \hfill {\text{Description}} \hfill \\
\hline 
\hline 
h_1 h_2 h_3 & h_1, h_2, h_3 \in \{k_1,k_2\} \\[5pt]
\hline 
h a a' & h \in \{k_1,k_2\} \:\:\: , \:\:\: A \in \cC \\[5pt]
\hline 
abc & \begin{array}{l}
A , B , C \in \cC \\[2pt]
A \cdot B = A \cdot C = B \cdot C = 2 \end{array}\\
\hline 
\end{array}
\end{equation}
Indeed, let $-3K_X=A+B+C$ be any expression of $-3K_X$ as a sum of three effective 
divisors.  If one among $A,B,C$ is in $|{-K}_X|$, then we are in one of the first two cases 
above.  If $A,B,C \in \cC$, then intersecting with $A,B,C$ successively both sides of the 
equation $-3K_X=A+B+C$ we obtain the system 
\[
\left\{ \begin{array} {rcl}
A \cdot B + A \cdot C & = & 4 \\[5pt]
A \cdot B + B \cdot C & = & 4 \\[5pt]
A \cdot C + B \cdot C & = & 4
\end{array} \right.
\]
whose only solution is $A \cdot B = A \cdot C = B \cdot C = 2$.

By definition $(I_X)_2 = (J_X)_2$; we deduce that the span of the monomials of the 
first two forms in (\ref{che}) in $k[\cG]/J_X$ has dimension at most six, being the image of $\HX{0}{-K_X} \otimes \HX{0}{-2K_X}$ under the 
multiplication map (note that $\dim \HX{0}{-K_X} = 2$, $\dim \HX{0}{-2K_X} = 4$ and 
the linearly independent elements $k_1 \otimes k_2^2 e - k_2 \otimes k_1k_2$ and $k_2 \otimes k_1^2 e - k_1 \otimes k_1k_2$ are in the kernel).  Fix any monomial $p = a b c$ of the third form in (\ref{che}) and denote the span of $p$ together with the monomials of the first two forms in (\ref{che}) 
by $V$.  Let $q = d e f$ be any monomial of the third form in (\ref{che}); we prove that 
$q$ can be written as the sum of an element of $V$ and an element of $J_X$.

The result is clear if $c =f$, since then we can use the relation coming from degree 
$A+B$ involving the monomials $k_1 g c$, $k_2 g c$, $a b c$ and $d e c$ to conclude (see Lemma~\ref{basis-K+G}).  By the same reasoning we reduce to the case in which $\{a,b,c\} \cap \{d,e,f\} = \emptyset$.

Suppose that $A \cdot D = 2$; then the divisor $G := -3K_X-A-D$ satisfies 
$G^2=K_X \cdot G = -1$, and hence $G$ is an exceptional curve.  Moreover from the 
fact that $D \cdot G = 2$, we deduce that $G$ is an exceptional curve on the del Pezzo 
surface obtained by contracting $D'$ and whose anticanonical divisor pulled-back to 
$X$ is $E+F$; let $H:=E+F-G$ and note that $H$ is also an exceptional curve.  In 
particular we can use a relation coming from degree $E+F$ involving the monomials 
$d k_1 d'$, $d k_2 d'$, $d g h$ and $d e f$ to reduce to the case 
$\{a,b,c\} \cap \{d,e,f\} \neq \emptyset$.

Finally, if $A \cdot D = A \cdot E = A \cdot F = 1$, then $A$ is a conic in the del Pezzo 
surface obtained by contracting $F'$.  From Table~\ref{tab:-1curves1} it follows that there are exceptional curves intersecting any given conic twice; denote the strict transform 
in $X$ of one such curve of $X$ by $H$ and the exceptional curve $E+F-H$ by $J$.  Thus we can 
use a relation coming from degree $E+F$ involving the monomials $d k_1 d'$, $d k_2 d'$, 
$d h j$ and $d e f$ to reduce to the case in which $A \cdot D = 2$.
\end{proof}

\begin{lemma} 
\label{duee}
Let $E$ be an exceptional curve on $X$; then the ideal $I_X$ has no minimal generators in degree 
$D=-2K_X+E$.
\end{lemma}

\begin{proof}
The monomials of degree $D$ are of the following forms:
\[
\begin{array} {|l|l|}
\hline 
{\text{Monomial}} & \hfill {\text{Description}} \hfill \\
\hline
\hline
k s & k \in \{k_1,k_2\} \,\,,\,\, s \in k[\cG]_{-K_X+E}\\
\hline
a a' e & A \cdot A' = 3 \\
\hline 
a b c & \begin{array}{l}
A \cdot E = 1 \\
B \cdot E = C \cdot E = 0 \\
A \cdot B = A \cdot C = 2 \\
B \cdot C = 1 \end{array} \\
\hline
\end{array}
\]
Indeed let $m \in k[\cG]_D$ be a monomial.  If $k \in \{k_1,k_2\}$, or $e$ divides $m$, then 
$m$ is of one of the first two forms.  Otherwise let $m=a b c$, with 
$\{a,b,c\} \cap \{k_1,k_2,e\} = \emptyset $; if two of the curves $A,B,C$ have intersection 
number three, then the remaining one is $E$, which we are excluding.   If the intersection 
numbers among the curves $A,B,C$ are all at most two, then the required conditions follow 
multiplying successively the equality $A+B+C=-2K_X+E$ by $A,B,C$.

We show that the monomials of degree $D$ span a subspace of dimension at most six of 
$(k[\cG]/J_X)_D$; since $\dim \Cox(X)_D = 6$, the result follows.

First, the image of $\HX{0}{-K_X} \otimes \HX{0}{-K_X+E}$ in $k[\cG]_D$ has dimension at 
most five: $\dim \HX{0}{-K_X} = 2$, $\dim \HX{0}{-K_X+E} = 3$ and 
the element $k_1 \otimes k_2 e - k_2 \otimes k_1 e$ is in the kernel.

Second, the span of the monomials of the first two forms has dimension at most six.  Let 
$a a' e , b b' e$ be any two monomials of the second form.  There is a quadratic relation $q$ 
involving $a a' , b b' , k_1^2 , k_1 k_2 , k_2^2$ since these five vectors correspond to five 
elements of the four dimensional space $\HX{0}{-2K_X}$.  Moreover $a a'$ and $b b'$ 
are independent from $k_1^2 , k_1 k_2 , k_2^2$, since the monomials in $k_1,k_2$ 
correspond to sections having a base-point, and neither of the remaining elements $a a'$ 
and $b b'$ vanishes at the base-point.  Thus in the relation $q$ the coefficients of $a a'$ and 
$b b'$ are both non-zero.  We deduce that the span of the elements of the first two forms has 
dimension at most six in $(k[\cG]/J_X)_D$.

Third, let $a b c$ be a monomial of the third form.  The divisor $A+B$ contracts the unique 
exceptional curve $F := K_X+A+B$, and the divisor $G := A+B-E$ is an exceptional curve 
on the del Pezzo surface $Y$ obtained from $X$ by contracting $F$.  Therefore there is a 
quadratic relation involving the monomials $a b , g e , k_1 f , k_2 f$ and the monomials 
$a b$ and $g e$ are independent from $k_1 f , k_2 f$, by Lemma~\ref{basis-K+G}.  Thus we may 
use this relation to write the image of $a b c$ in $k[\cG]/J_X$ as a combination of divisors 
of the first two forms and the proof is complete.
\end{proof}

\begin{lemma} 
\label{kaco}
Let $Q$ be a conic on $X$; then the ideal $I_X$ has no minimal generators in degree $D=-K_X+Q$.
\end{lemma}

\begin{proof}
The monomials of degree $D$ are of the following forms:
\[
\begin{array} {|l|l|}
\hline 
{\text{Monomial}} & \hfill {\text{Description}} \hfill \\
\hline
\hline
k s & k \in \{k_1,k_2\} \,\,,\,\, s \in k[\cG]_Q \\
\hline
e f e_1 & \begin{array} {l}
E_1 \cdot Q = 0 \\
E \cdot F = 2 \\
E_1 \cdot (E+F) = 2 \\
Q \cdot E = Q \cdot F = 1 \end{array} \\
\hline 
e e_1 e_2 & \begin{array}{l} 
E_1 \cdot Q = E_2 \cdot Q = E_1 \cdot E_2 = 0 \\
Q \cdot E = E_1 \cdot E = E_2 \cdot E = 2 \end{array} \\
\hline
\end{array}
\]
Note that $E_1,E_2$ are disjoint components of reducible fibers of the conic 
bundle associated to $Q$.  Indeed let 
\begin{equation} \label{primi}
-K_X+Q \sim A+B+C
\end{equation}
Multiplying both sides of (\ref{primi}) by $Q$ we find 
\[
2 = A \cdot Q + B \cdot Q + C \cdot Q ;
\]
thus at least one of the intersection products is zero.  Suppose that $C \cdot Q = 0$, 
and let $\tilde C = Q-C$; then $A+B = -K_X+\tilde C$ is a reducible divisor in the linear 
system associated to the pull-back of the anticanonical divisor on the del Pezzo surface 
of degree two obtained from $X$ by contracting $\tilde E$.  The first case in 
(\ref{primi}) corresponds to the sections containing a divisor in the linear system 
$|{-K}_X|$; the second to one not containing it and containing two curves not in the 
linear system $|Q|$; the last one to one containing two elements of $|Q|$.

Since $(k[\cG]/J_X)_2 = \Cox (X)_2$, it follows that the dimension of the span 
of the monomials of the first form in (\ref{primi}) modulo $(J_X)_{-K_X+Q}$ is at most 
four.  Let $p = e e_1 e_2$ be a monomial in $k[\cG]_{-K_X+Q}$ of the third form in 
(\ref{primi}) and let $V$ denote the span of the monomials of the first form in (\ref{primi}) 
together with $p$ in $(k[\cG]/J_X)_{-K_X+Q}$.  Since $\dim (\Cox (X)_{-K_X+Q}) = 5$, 
the result follows if we show that $V = (k[\cG]/J_X)_{-K_X+Q}$.

Let $q = \bar e \bar e_1 \bar e_2$ be a monomial in $k[\cG]_{-K_X+Q}$ of the third 
form in (\ref{primi}).

If $\bar e_1 = e_1$, then let $E_3 = K_X+E+E_2=K_X+\bar E+ \bar E_2$ be the 
exceptional curve contracted by $E+E_2$.  By Lemma~\ref{basis-K+G}, applied to 
the monomials $\bar e \bar e_2 , e e_2 , k_1 e_3 , k_2 e_3$, we conclude that the  
image of $q$ in $(k[\cG]/J_X)_{-K_X+Q}$ belongs to $V$.

If $\{ \bar E_1 , \bar E_2 \} \cap \{ E_1 , E_2 \} = \emptyset $, then at least one curve in 
$\{ \bar E_1 , \bar E_2 \}$ is disjoint from one curve in $\{ E_1 , E_2 \}$ and relabeling the 
indices if necessary we may assume that $\bar E_1 \cdot E_1 = 0$.  By the same reasoning 
above $q = \bar e \bar e_1 \bar e_2$ is in the span of $V$ and $\tilde q = \tilde e \bar e_1 e_1$ 
(note that $\tilde E := -K_X+Q-\bar E_1 - E_1$ is an exceptional curve, since it has anticanonical 
degree one and square negative one), and we conclude since $\tilde q$ is in the span of $V$.

Finally, if $q = \bar e \bar f \bar e_1$, then at least one curve in $E_1 , E_2$ is disjoint from 
$\bar E_1$ and relabeling the indices if necessary we may assume that $\bar E_1 \cdot E_1 = 0$.  
Reasoning as above, $q$ is in the span of $V$ and $\tilde e e_1 \bar e_1$, and we are done.
\end{proof}

\begin{cor} \label{ganzo}
Let $X$ be a {\free} del Pezzo surface of degree one and let $D$ be an ample divisor on $X$ 
such that $-K_X \cdot D \geq 3$.  Then the ideal $I_X$ has no minimal generators in degree $D$.
\end{cor}

\begin{proof}
By Theorem~\ref{non4}, if $-K_X \cdot D \geq 4$ then $D$ is capturable and the result follows from Theorem~\ref{thm:games}.  Since the only ample divisors of anticanonical three are $-3K_X$, $-2K_X+E$ and $-K_X+Q$, where $E$ is an exceptional curve and $Q$ is a conic, the result is true if 
$-K_X \cdot D = 3$, by Lemmas~\ref{tretre}, \ref{duee} and~\ref{kaco}.
\end{proof}

\section{Non ample divisors}
\label{S:non ample divisors}

In this section we prove that the relations in degree $D$, where $D$ is a non ample 
divisor on $X$ of anticanonical degree at least three, coming from $J_X$ are 
sufficient to show that the quotient $(k[\cG]/J_X)_D$ is in fact spanned by monomials 
coming from a del Pezzo surface of smaller degree.  This is the basis for the inductive 
procedure of \S\ref{genedue}.  Throughout this section we assume that del Pezzo surfaces $X$ of degree one defined over fields of characteristic two are {\free}.

\begin{lemma} \label{tirati}
Let $X$ be a del Pezzo surface of degree at most five and let $D$ be a divisor on $X$.  
Suppose that $E$ is an exceptional curve on $X$ such that $D \cdot E = 0$.  Then 
$(k[\cG]/J_X )_D$ is spanned by products of variables corresponding to exceptional 
curves disjoint from $E$.
\end{lemma}

\begin{proof}
Let $m \in k[\cG]_D$ be a monomial and write $m = e^p \cdot s$, where $p \geq 0$ 
and $s$ is a product of variables different from $E$.  Note that if $p = 0$, then $s$ is a 
product of variables corresponding to divisors disjoint from $E$, since by assumption 
$D \cdot E = 0$.  We shall show that if $p \geq 1$, then using the quadratic relations 
we may decrease $p$; the result then follows by induction on $p$.

Suppose that $p\geq 1$.  Since $D \cdot E = 0$, there is a variable $c \in \cG$ such that 
$E \cdot C > 0$ and $c$ divides $s$.
The monomial $ce$ is a monomial of anticanonical degree two corresponding to a nef 
divisor.  By definition of $J_X$, the vector space $(k[\cG]/J_X)_{C+E}$ coincides with 
$\Cox (X)_{C+E}$; thus it suffices to show that there is a basis of $\Cox (X)_{C+E}$ 
consisting of sections vanishing along exceptional curves different from $E$.  To conclude, we analyze all the possibilities for the divisor $C+E$.

\noindent
{\bf Case 1:} The divisor $C+E$ is a conic.  The linear system $|C+E|$ contains $8-\deg (X) \geq 3$ distinct reducible elements and any two of these span it.

\noindent
{\bf Case 2:} The divisor $C+E$ is $-b^*K_Y$, where $b \colon X \to Y$ is a birational morphism, $Y$ has degree two.  The linear system $|C+E|$ contains $28$ distinct reducible elements and any five of these span it (Lemma~\ref{gene2}).

\noindent
{\bf Case 3:} The divisor $C+E$ is ${-2K_X}$ and the degree of $X$ is one.  The linear system $|{-2K_X}|$ contains $120$ distinct reducible elements whose irreducible components are $(-1)$-curves and any $113$ of these span it (Proposition~\ref{tritangenti}).

In all cases the sections supported on $(-1)$-curves distinct from $E$ span $\Cox (X)_{C+E}$ and the lemma follows.
\end{proof}

\begin{lemma} 
\label{negati}
Let $X$ be a del Pezzo surface of degree at most five and let $D$ be a divisor on $X$.  
Suppose that $E$ is an exceptional curve on $X$ such that $D \cdot E < 0$.  Then 
$\dim (k[\cG]/J_X )_D = \dim (k[\cG]/J_X )_{D-E}$ and 
$\dim \Cox (X)_D = \dim \Cox (X)_{D-E}$.
\end{lemma}

\begin{proof}
If $D$ is not effective, then $D-E$ is not effective and $k[\cG]_D = k[\cG]_{D-E} = 0$.  If 
$D$ is effective, then $E \cdot D < 0$ implies that $E$ is a component of $D$ and thus 
every monomial in $k[\cG]_D$ is divisible by $e$; the same argument also proves the 
statement about the Cox ring.
\end{proof}

\section{Del Pezzo surfaces of degree at least two} \label{Xgrado2}

\begin{lemma} 
\label{parti2}
Let $X$ be a del Pezzo surface of degree two.  Any ample divisor is capturable 
except $-K_X$.
\end{lemma}

\begin{proof}
Let $D \neq -K_X$ be an ample divisor and write $D = -nK_X +N$, with $n \geq 1$ and 
$N$ nef and not ample; by assumption either $n \geq 2$ or $N \neq 0$.

Suppose first that $N \neq 0$; if $N$ contracts more than one exceptional curve, then we let $\cS$ 
be the set of exceptional curves contracted by $N$. If $N$ contracts exactly one exceptional 
curve $E$, then we let $F$ be any exceptional curve disjoint from $E$ and let $\cS := \{E,F\}$.

Let $C \in \cC \setminus \cS$ and let $C' := -K_X-C$.  If $C' \in \cS$, then the divisor 
$L_{C'T} = N+(-K_X-T)-(n-1)K_X$ is big and nef for all $T \in \cS$, $T \neq C'$ since 
$(-K_X-T) \cdot N = -K_X \cdot N > 0$; thus we can capture all such curves using $\cS$.  
If $C' \notin \cS$, then let $S,T \in \cS$ be disjoint and note that $N-T$ is either nef or 
$N$ is a multiple of a conic and it is the sum of a multiple of a conic and a single exceptional 
curve.  Moreover 
\begin{itemize}
\item $S' \cdot (N-T) = -K_X \cdot N-1>0$, 
\item if $N$ is a multiple of a conic, then $S'$ also intersects the fixed component of $N-T$, 
\item $C' \cdot (N-T) = -K_X \cdot N - 1 - C \cdot T$ is zero if and only if $N$ is a conic and 
$C \cdot T = 1$.
\end{itemize}
In this last case, note that $C \cdot N \neq 0$, and $C$ intersects at least one component in 
every reducible fiber of the conic bundle determined by $N$; hence we may choose $T$ to 
be the other component if $C' \cdot (N-T) = 0$.  In all cases we may choose $S,T \in \cS$ 
disjoint such that $L_{ST}$ is big and nef, and hence capture all $C \in \cC \setminus \cS$ 
using only curves in $\cS$.  In particular, we are done if $N$ contracts at most two exceptional 
curves.

To conclude it suffices to treat the two cases $N=L+N'$ and $N = Q+N'$, where $L$ is a 
twisted cubic, $Q$ is a conic and $N'$ is nef; note also that $S \cdot N' = 0$, for all $S \in \cS$.

Choose any two disjoint $S,T \in \cS$ and let $C \in \cS$, $C \neq S,T$.  
If $N=L+N'$, we have that the divisor $L_{ST} = -K_X+(-K_X-C)+(L-S-T)+N'$ is big and 
nef since $(-K_X-C) \cdot (L-S-T) = 1$ and we conclude.  If $N=Q+N'$, then the 
divisor $L_{ST} = (-K_X-S)+(-K_X-T)+(Q-C)+N'$ is big and nef since 
$(-K_X-S) \cdot (Q-C) = (-K_X-T) \cdot (Q-C) = 1$ and again we conclude.

Suppose now that $N=0$ and hence $D = -nK_X$, where $n \geq 2$.  Let $\cS$ be any 
set of seven disjoint exceptional curves.  For all $C \in \cC \setminus \cS$, there are 
distinct $S,T \in \cS$ such that $C \cdot S , C \cdot T > 0$: this is clear from the list of 
exceptional curves.  Thus for such a choice of $S,T$, we have that the divisor 
$L_{ST} = -3K_X-C-S-T = C'+S'+T'$ is big and nef and we may capture all the curves 
in $\cC \setminus \cS$ using $\cS$.  Finally let $S,T \in \cS$ be distinct elements and let 
$S' := -K_X-S$, $T' := -K_X-T$; for all $C \in \cS$ the divisor 
$L_{S'T'} = -3K_X-C-S'-T' = C'+S+T$ is big and nef and we conclude.
\end{proof}

\begin{lemma} 
\label{parti3}
Let $X$ be a del Pezzo surface of degree three.  Any ample divisor is capturable.
\end{lemma}

\begin{proof}
Let $D$ be an ample divisor and write $D = -K_X +N$, with $N$ a nef divisor.

Let $\cS$ be a set of six disjoint exceptional curves.  For all $C \in \cC \setminus \cS$ 
there are two distinct $S,T \in \cS$ such that $C \cdot S = C \cdot T = 1$.  With these 
choices of $S,T$ we have that the divisor $L_{ST} = -2K_X+N-C-S-T = (-K_X-T)+(-K_X-C-S)+N$ 
is big and nef since $(-K_X-T) \cdot (-K_X-C-S) = 1$.  Thus we may capture all the 
exceptional curves $C \in \cC \setminus \cS$ using the curves in $\cS$.  Let $\cS'$ be 
the set of six disjoint exceptional curves each intersecting all the curves in $\cS$ 
except for one (thus $\cS \cup \cS'$ is a ``Schl\"afli double-six'').  Choose 
$S,T \in \cS$ and let $S',T' \in \cS'$ be two corresponding curves, that is 
$S \cdot S' = T \cdot T' = 0$.  Using the moves above we can capture $\{ S',T' \}$ with 
$\cS$ and we can capture $\cS \setminus \{ S,T \}$ using $\cS' \setminus \{ S',T' \}$, and 
finally we capture $\cS' \setminus \{ S',T' \}$ using $\{ S,T \}$.
\end{proof}

\begin{lemma} \label{parti4}
Let $X$ be a del Pezzo surface of degree four.  Any ample divisor is capturable.
\end{lemma}

\begin{proof}
Let $D$ be an ample divisor and write $D = -K_X +N$, with $N$ a nef divisor.  
The criterion to capture curves is that $(\{E,F\} , C)$ is a capture move if $C \notin \{E,F\}$, 
$E \cdot F = 0$ and $C \cdot E + C \cdot F > 0$: if $C \cdot E > 0$, the divisor 
$L_{EF} = (-K_X-C-E)+(-K_X-F)+N$ is the sum of a conic, a twisted cubic and a nef divisor 
and hence it is big and nef.

We refer to Table~\ref{tab:-1curves1}.  Use $E_3,E_4$ to capture $L-E_i-E_j$ for all 
$3 \leq i < j < 5$.  Use $L-E_1-E_3 , L-E_1-E_4 , L-E_1-E_5$ to capture $E_3,E_4,E_5$ 
and finally use $E_1,E_2$ to capture all the remaining curves.
\end{proof}

\begin{lemma} 
\label{parti5}
Let $X$ be a del Pezzo surface of degree five.  Any ample divisor is capturable.
\end{lemma}

\begin{proof}
Same criterion and strategy as in Lemma~\ref{parti4}, ignoring any reference to the 
index 5.
\end{proof}

\begin{lemma} \label{parti6}
Let $X$ be a del Pezzo surface of degree six.  Any ample divisor is capturable.
\end{lemma}

\begin{proof}
Same criterion as in Lemma~\ref{parti4}.  Capture $L-E_1-E_2$ using $E_1,E_2$ and 
then capture the remaining curves using $L-E_1-E_3 , L-E_2-E_3$.
\end{proof}

\section{Quadratic generation} 
\label{genedue}

\subsection{Main result}
We collect all the information gathered in the previous sections to prove the main 
result of this paper: the ideal $I_X$ is generated by its degree two part.

\begin{proof}[Proof of the Batyrev--Popov conjecture]
Let $n$ be an integer; by induction on $r_X := 9 - \deg (X)$ and by induction on $n$ 
we show that $(J_X)_n = (I_X)_n$.  By definition, the statement is true if $n \leq 2$, for 
all del Pezzo surfaces $X$; if $r_X \leq 4$, then the result is well-known for all $n$.

Suppose that $r_X > 4$, $n \geq 3$, that for all del Pezzo surfaces $Y$ such that 
$r _Y < r _X$ we have that $J_Y=I_Y$, and that $(J_X)_{n-1} = (I_X)_{n-1}$.  
Let $D$ be a divisor on $X$ of anticanonical degree $n$; if there is an exceptional curve 
$E$ such that $D \cdot E \leq 0$, then the result follows by Lemma~\ref{negati} or 
by Lemma~\ref{tirati}.  Otherwise $D$ is ample and the result follows from 
Theorem~\ref{ganzo} if $\deg (X) = 1$, or from the lemmas of Section~\ref{Xgrado2} if 
$\deg (X) \geq 2$.
\end{proof}

\subsection{Quadratic generators} \label{ulti}

Let $X$ be a del Pezzo surface of degree $d$. We briefly explain how all generators 
of $I_{X}$ arise.  The nef divisors $D$ with anticanonical degree two on $X$ are:
\begin{enumerate}
\item conics $Q$, 
\item ${-K}_X$ if $d = 2$,
\item $-K_X + E$ if $d=1$, where $E$ is an exceptional curve on $X$,
\item $-2K_{X}$ ($d = 1$).
\end{enumerate}
These divisors are precisely sums of pairs of intersecting exceptional curves on $X$.

We count the relations coming from conics as follows. Every conic  $Q$ has $8 - d$ reducible sections and $\hX{0}{Q} = 2$; thus each conic gives rise to $6 - d$ quadratic relations. If there are $B$ conics on $X$ then there are $(6-d)B$ generators of $I_X$ induced by conic bundles. For example, when $d = 1$ we obtain $5\times 2,\!160 = 10,\!800$ relations.

When $d = 2$, we also have relations in degree $-K_{X}$. There are 28 monomials in $k[\cG]_{-K_X}$, and $\hX{0}{-K_X} = 3$, giving 25 linear dependence relations among these quadratic monomials. These relations yield the remaining generators of $I_{X}$. 

When $d=1$, there are $30$ monomials in $k[\cG]_{-K_X + E}$: $k_1e$, $k_2e$ and the $28$ monomials coming from the anticanonical divisor of the del Pezzo surface of degree two obtained by contracting $E$. This gives $(30 - 3)\times 240 = 6,\!480$ generators in $I_{X}$. Finally, we also have relations coming from $-2K_{X}$.  There are 123 monomials in $k[\cG]_{-2K_X}$: $k_1^2, k_1k_2$, $k_2^2$ and the $120$ monomials of the form $ee'$, where $E,E' \in \cC$ and $E + E' = -2K_X$. Since $\hX{0}{-2K_X} = 4$, we obtain 119 linear dependence relations among these quadratic monomials. 
These relations yield the remaining generators of $I_{X}$. This information is summarized in Table~\ref{Ta:relations}.

\begin{table}
\begin{center}
\caption{First Betti numbers $b_{1,D}(\Cox(X))$ for del Pezzo surfaces}
\label{Ta:relations}
\begin{tabular}{|c|c|c|c|c|c|}
\hline
$\deg(X)$ & $D$ & \begin{tabular}{c} Number of \\ monomials \\ in $k[\cG]_D$ \end{tabular} 
& $\hX{0}{D}$ & $b_{1,D}(\Cox(X))$ & 
\begin{tabular}{c} Number of \\ divisors \\ of type $D$ \end{tabular} \\
\hline 
\hline 
$1$ & $Q$ & $7$ & $2$ & $5$ & $2160$ \\
& $-K_{X} + E$ & $28+2$ & $3$ & $27$ & $240$ \\
& $-2K_X$ & $120+3$ & $4$ & $119$ & $1$ \\
 & (Total) & $22443$ &  & $17399$ & $2401$ \\
\hline
$2$ & $Q$ & $6$ & $2$ & $4$ & $126$ \\ 
& $-K_X$ & $28$ & $3$ & $25$ & $1$ \\
 & (Total) & $784$ &  & $529$ & $127$ \\
\hline
$3$ & $Q$ & $5$ & $2$ & $3$ & $27$ \\
 & (Total) & $135$ &  & $81$ & $27$ \\
\hline 
$4$ & $Q$ & $4$ & $2$ & $2$ & $10$ \\
 & (Total) & $40$ &  & $20$ & $10$ \\
\hline 
$5$ & $Q$ & $3$ & $2$ & $1$ & $5$ \\
 & (Total) & $15$ &  & $5$ & $5$ \\
\hline
\end{tabular}
\end{center}
\medskip
\end{table}

\end{document}